\PassOptionsToPackage{unicode=true}{hyperref} 
\documentclass[a4paper]{article}
\usepackage{amssymb,amsmath}

\usepackage{hyperref}

\makeatletter
\def\fps@figure{htbp}
\makeatother

\usepackage{amsthm}
\usepackage{bussproofs}
\EnableBpAbbreviations 

\usepackage{pdflscape}
\usepackage{afterpage}
\usepackage{longtable}

\usepackage{slashed}

\usepackage{pdfpages}

\newtheorem{theorem}{Theorem}
\newtheorem{conjecture}[theorem]{Conjecture}
\newtheorem{lemma}[theorem]{Lemma}
\newtheorem{corollary}[theorem]{Corollary}
\newtheorem{lcorollary}[theorem]{Corollary}
\theoremstyle{definition}
\newtheorem{definition}[theorem]{Definition}

\newcommand\uhp{\mathord{\upharpoonright}}

\newcommand{\GassBox}{\ensuremath{\Box}}
\newcommand{\GassDiamond}{\Diamond}
\newcommand{\GassLeft}{\ensuremath{\Leftarrow}}
\newcommand{\GassRight}{\ensuremath{\Rightarrow}}

\newcommand{\GassCircli}{^\circ}
\newcommand{\GassT}{\ensuremath{\mathcal T}}

\DeclareMathOperator{\Glp}{G3lp}
\DeclareMathOperator{\Gs}{G3s}
\DeclareMathOperator{\Cut}{Cut}
\DeclareMathOperator{\sub}{sub}
\DeclareMathOperator{\IN}{IN}
\DeclareMathOperator{\CS}{CS}
\DeclareMathOperator{\LP}{LP}
\DeclareMathOperator{\LPG}{LPG}
\DeclareMathOperator{\an}{\mathsf{an}}

\makeatletter
\newcommand{\todo}[1]{\marginpar{\textbf{TODO\footnotemark}}\@latex@warning{TODO: #1}\footnotetext{ #1}}
\makeatother

\title{Self-referentiality in Justification Logic}
\author{Nathan Sebastian Gass \and Thomas Studer}
\date{}

\begin{document}
\maketitle

\framebox{
\begin{minipage}{\linewidth}
\textbf{Disclaimer.} 
This paper contains several mistakes. A revised version is under progress. 

\par\medskip

Note, however, that the counterexample in Section 4.3 is valid. It is correct that prehistoric cycles in G3s are not sufficient for needing a self-referential CS.

\end{minipage}
}

\begin{abstract}
The Logic of Proofs, LP, and other justification logics can have
self-referential justifications of the form \(t{:}A(t)\). Such
self-referential justifications are necessary for the realization of S4
in LP.  Yu discovered prehistoric cycles in a particular Gentzen
system as a necessary condition for S4 theorems that can only be
realized using self-referentiality. 
It was an open problem whether prehistoric cycles also are a sufficient condition.

The main results of this paper are:
First, with the standard definition of
self-referential theorems, prehistoric cycles are not a sufficient
condition. Second, with an expansion on that definition, prehistoric
cycles become sufficient for self-referential theorems.
\end{abstract}

\section{Introduction}\label{introduction}

Sergei Artemov~\cite{artemov1995,artemov2001} first introduced the Logic of Proofs, LP.
He replaced the  \(\GassBox \)-modality of S4 with explicit proof terms to obtain a classical provability semantics for intuitionistic logic.
Later more applications
of this explicit notations were discovered for different epistemic
logics~\cite{brezhnev2001}. So Artemov~\cite{artemov2008} introduced the more general
notion of justification logics where justification terms take over the
role of the proof terms in LP. In any justification logic \(t{:}A\) is
read as \(t\) is a justification of \(A\), leaving open what exactly
that entails. Using different axioms and different operators, various
different justification logic counterparts where developed for the
different modal systems used in epistemic logic (K, T, K4, S4, K45,
KD45, S5, etc.).

In justification logics it is possible for a term \(t\) to be a
justification for a formula \(A(t)\) containing \(t\) itself, i.e.~for
the assertion \(t{:}A(t)\) to hold. Prima facie this seems suspicious
from a philosophical standpoint as well for more formal mathematical
reasons. Such a self-referential sentence is for example impossible with
an arithmetic proof predicate using standard G\"odel numbers as the G\"odel
number of a proof is always greater than any number referenced in it as
discussed by Roman Kuznets~\cite{kuznets2010}.
In the same paper, the author argues that there is nothing inherently
wrong with self-referential justifications if we understand the
justifications as valid reasoning templates or schemes, which of course
then can be used on themselves.

Kuznets studied the topic of self-referentiality at the logic-level. He
discovered theorems of S4, D4, T and K4 that need a self-referential
constant specification to be realized in their justification logic
counterparts~\cite{kuznets2010}. Junhua
Yu on the other hand studied self-referentiality at the theorem level.
He discovered prehistoric cycles as a necessary condition for
self-referential S4 theorems~\cite{yu2010} 
and later expanded that results to the modal logics T and K4~\cite{yu2014}. 
He also conjectured that the
condition is actually sufficient for self-referential S4 theorems. In
this paper we will concentrate on that topic, that is prehistoric cycles
as necessary and sufficient condition for self-referential theorems in
S4.

This paper is divided in three parts. In the first part we introduce the
modal logic S4 and its justification counterpart LP. The second part
restates Yu's main theorem, i.e.~that prehistoric cycles are a
necessary condition for self-referential theorems in S4. The third part
goes beyond Yu's original paper by adapting the notion of prehistoric
cycles to Gentzen systems with cut rules and finally to a Gentzen system
for LP. This allows to study prehistoric cycles directly in LP, which
leads to the two main results of this paper. First, with the standard
definition of self-referential theorems, prehistoric cycles are not a
sufficient condition. Second, with an expansion on the definition of
self-referential theorems, prehistoric cycles become sufficient for
self-referential theorems.

\section{LP and S4}\label{lp-and-s4}

\subsection{Preliminaries}\label{preliminaries}

As the results and concepts in this paper are mostly purely syntactical,
we will also limit this brief introduction to the modal logic S4 and its
justification counterpart LP to the syntactic side. 

\begin{definition}[Syntax of S4] The language of S4 is given by
\(A := \bot \mid P \mid A_0 \land  A_1 \mid A_0 \lor  A_1 \mid A_0 \to  A_1 \mid \GassBox A \mid \GassDiamond A\). By using
the known abbreviations for \(\land \), \(\lor \) and \(\GassDiamond \) we can reduce
that to the minimal language \(A := \bot  \mid P \mid A_0 \to  A_1 \mid \GassBox A\).
\end{definition}

\begin{definition}[Syntax of LP] The language of LP consists of
terms given by \(t := c \mid x \mid t_0 \cdot  t_1 \mid t_0 + t_0 \mid\: !t\) and
formulas given by \(A := \bot  \mid P \mid A_0 \to  A_1 \mid t{:}A\). \end{definition}

A Hilbert style system for LP is given by the following Axioms and the
rules modus ponens and axiom necessitation~\cite{artemov2001}

\begin{itemize}
\item
  \(A0\): Finite set of axiom schemes of classical propositional logic
\item
  \(A1\): \(t{:}F \to  F\) (Reflection)
\item
  \(A2\): \(s{:}(F \to  G) \to  (t{:}F \to  (s\cdot t){:}G)\) (Application)
\item
  \(A3\): \(t{:}F \to \;!t{:}(t{:}F)\) (Proof Checker)
\item
  \(A4\): \(s{:}F \to  (s+t){:}F\), \(t{:}F \to  (s+t){:}F\) (Sum)
\end{itemize}

\begin{itemize}
\item
  \(R1\): \(F \to  G, F \vdash  G\) (Modus Ponens)
\item
  \(R2\): \(A \vdash  c{:}A\), if \(A\) is an axiom \(A0-A4\) and \(c\) a
  constant (Axiom Necessitation)
\end{itemize}

A Hilbert style derivation \(d\) from a set of assumptions \(\Gamma \) is a
sequence of formulas \(A_0, \ldots, A_n\) such that any formula is either an
instance of an axiom A0--A4, a formula \(A \in  \Gamma \) or derived from earlier
formulas by a rule R1 or R2. The notation \(\Gamma  \vdash _{\LP} A\) means that a
LP derivation from assumptions \(\Gamma \) ending in \(A\) exists. We also
write \(\vdash _{\LP} A\) or \(\LP \vdash  A\) if a LP derivation for \(A\) without
any assumptions exists.

When formulating such derivations, we will introduce propositional
tautologies without derivation and use the term propositional reasoning
for any use of modus ponens together with a propositional tautology.
This is of course correct as axioms A0 together with the modus ponens
rule R1 are a complete Hilbert style system for classical propositional
logic. Its easy to see by a simple complete induction on the proof
length that this derivations do not use any new terms not already
occurring in the final propositional tautology.

\begin{definition}[Constant Specification] A \emph{constant
specification} CS is a set of of formulas of the form \(c{:}A\) with
\(c\) a constant and \(A\) an axiom A0-A4. \end{definition}

Every LP derivation naturally generates a finite constant specification
of all formulas derived by axiom necessitation (R2). For a given
constant specification~CS, LP(CS) is the logic with axiom necessitation
restricted to that CS. \(\LP_0 := \LP(\emptyset)\) is the logic without axiom
necessitation. A constant specification CS is injective if for each
constant \(c\) there is at most one formula \(c{:}A \in  \CS\).

\subsection{Gentzen Systems}\label{gentzen-systems}

In the following, capital greek letters \(\Gamma \), \(\Delta \) are used for
multisets of formulas, latin letters \(P\), \(Q\) for atomic formulas
and latin letters \(A\), \(B\) for arbitrary formulas. We also use the
following short forms:

\begin{itemize}
\item
  \(\GassBox \Gamma  := \{\GassBox A | A \in  \Gamma \}\)
\item
  \(\Gamma ,A := \Gamma  \cup  \{A\}\)
\item
  \(\Gamma ,\Delta  := \Gamma  \cup  \Delta \)
\item
  \( \bigwedge \Gamma  := A_0 \land  \cdots \land  A_n\) and \( \bigvee \Gamma  := A_0 \lor  \cdots \lor  A_n\) for the
  formulas \(A_i \in  \Gamma \) in an arbitrary but fixed order.
\end{itemize}

Throughout this paper, we will use the G3s calculus from Troelstra and
Schwichtenberg~\cite{troelstra2000} for
our examples with additional rules \((\lnot \supset )\) and \((\supset \lnot )\) as we are only
concerned with classical logic (see figure \ref{G3sfull}). For proofs, on
the other hand, we  use a minimal subset of that system 
consisting only of \((Ax)\), \((\bot \supset )\), \((\to  \supset )\), \((\supset \to )\), \((\GassBox \supset )\), and \((\supset \GassBox)\) 
using the standard derived definitions for \(\lnot \),
\(\lor \), \(\land \) and \(\GassDiamond \). 

\renewcommand{\arraystretch}{3}
\begin{figure} \caption{Full G3s} \label{G3sfull}
\begin{longtable}{cc}

\AXC{$P, \Gamma  \supset  \Delta , P$ $(Ax)$ ($P$ atomic)}
\DP

&

\AXC{$\bot , \Gamma  \supset  \Delta $ $(\bot \supset )$}
\DP

\\

\RightLabel{$(\lnot  \supset )$}
\AXC{$\Gamma  \supset  \Delta , A$}
\UIC{$\lnot A, \Gamma  \supset  \Delta $}
\DP

&

\RightLabel{$(\supset  \lnot )$}
\AXC{$A, \Gamma  \supset  \Delta $}
\UIC{$\Gamma  \supset  \Delta , \lnot A$}
\DP

\\

\RightLabel{$(\land  \supset )$}
\AXC{$A, B, \Gamma  \supset  \Delta $}
\UIC{$A \land  B, \Gamma  \supset  \Delta $}
\DP

&

\RightLabel{$(\supset  \land )$}
\AXC{$\Gamma  \supset  \Delta , A$}
\AXC{$\Gamma  \supset  \Delta , B$}
\BIC{$\Gamma  \supset  \Delta , A \land  B$}
\DP

\\

\RightLabel{$(\lor  \supset )$}
\AXC{$A, \Gamma  \supset  \Delta $}
\AXC{$B, \Gamma  \supset  \Delta $}
\BIC{$A \lor  B, \Gamma  \supset  \Delta $}
\DP

&

\RightLabel{$(\supset  \lor )$}
\AXC{$\Gamma  \supset  \Delta , A, B$}
\UIC{$\Gamma  \supset  \Delta , A \lor  B$}
\DP

\\

\RightLabel{$(\to  \supset )$}
\AXC{$\Gamma  \supset  \Delta , A$}
\AXC{$B, \Gamma  \supset  \Delta $}
\BIC{$A \to  B, \Gamma  \supset  \Delta $}
\DP

&

\RightLabel{$(\supset  \to )$}
\AXC{$A,\Gamma  \supset  \Delta , B$}
\UIC{$\Gamma  \supset  \Delta , A \to  B$}
\DP

\\

\RightLabel{$(\GassBox  \supset )$}
\AXC{$A, \GassBox A, \Gamma  \supset  \Delta $}
\UIC{$\GassBox A, \Gamma \supset  \Delta $}
\DP

&

\RightLabel{$(\supset  \GassBox )$}
\AXC{$\GassBox \Gamma  \supset  \GassDiamond \Delta , A$}
\UIC{$\Gamma ', \GassBox \Gamma   \supset  \GassDiamond \Delta , \Delta ', \GassBox A$}
\DP

\\

\RightLabel{$(\GassDiamond  \supset )$}
\AXC{$A, \GassBox \Gamma  \supset  \GassDiamond \Delta $}
\UIC{$\GassDiamond A, \Gamma ', \GassBox \Gamma  \supset  \GassDiamond \Delta , \Delta '$}
\DP

&

\RightLabel{$(\supset  \GassDiamond )$}
\AXC{$\Gamma  \supset  \Delta , A, \GassDiamond A$}
\UIC{$\Gamma  \supset  \Delta , \GassDiamond A$}
\DP

\end{longtable}
\end{figure}

In Artemov~\cite{artemov2001}, a
Gentzen-Style system LPG is introduced for the logic of proofs LP using
explicit contraction and weakening rules, i.e.~based on G1c as defined
in Troelstra and Schwichtenberg~\cite{troelstra2000}. Later we will follow
Cornelia Pulver~\cite{pulver2010} instead and
use G3lp with the structural rules absorbed.

In all rules, arbitrary formulas which occur in the premises and
the conclusion (denoted by repeated multisets \(\Gamma \), \(\GassBox \Gamma \), \(\Delta \) and
\(\GassDiamond \Delta \)) are called side formulas. Arbitrary formulas which only occur in
the conclusion (denoted by new multisets \(\Gamma \), \(\Delta \), \(\Gamma '\), \(\Delta '\))
are called weakening formulas.\footnote{Notice that weakening formulas
  only occur in axioms and the rules \((\supset  \GassBox )\) and \((\GassDiamond  \supset )\),
  which are also the only rules that restrict the possible side
  formulas.} The remaining single new formula in the conclusion is
called the principal formula of the rule. The remaining formulas in the
premises are called active formulas. Active formulas are always used as
subformulas of the principal formula. Active formulas which are also
strict subformulas of other active formulas of the same rule as used in
\((: \supset )\) and \((\GassBox  \supset )\) are contraction formulas.

Formally, a Gentzen style proof is denoted by \(\GassT  = (T, R)\), where
\(T := \{S_0, \ldots, S_n\}\) is the set of occurrences of sequents, and
\begin{multline*}
R := \{(S_i,S_j) \in  T \times  T \mid\\ 
\text{$S_i$ is the conclusion of a rule which has $S_j$ as a premise}\}.
\end{multline*}
The only root sequent of \(\GassT \) is denoted
by \(S_r\). A leaf sequent \(S\) is a sequent without any premises, i.e
\(S \slashed{R} S'\) for all \(S' \in  T\).

A path in a proof tree is a list of related sequent occurrences
\(S_0 R \ldots R S_n\). A root path is a path starting at the root sequent
\(S_r\). A root-leaf path is a root path ending in a leaf
sequent.\footnote{Yu uses the term path for a root path and branch for a
  root-leaf path. As this terminology is ambiguous we adopted the
  slightly different terminology given here.} A root path is fully
defined by the last sequent~\(S\). So we will use root path \(S\) to
mean the unique path \(S_r R S_0 R \ldots R S\) from the root \(S_r\) to
the sequent \(S\). \(T\uhp S\) denotes the subtree of \(T\) with root \(S\).
The transitive closure of \(R\) is denoted by \(R^+\) and the
reflexive-transitive closure is denoted by \(R^*\).

Consistent with the notation for the Hilbert style system LP, the
notation \(G \vdash  \Gamma  \subset  \Delta \) is used if there exists a Gentzen style proof
tree with the sequent \(\Gamma  \subset  \Delta \) as root in the system \(G\).

\begin{definition}[Correspondence] \label{corr} The subformula
(symbol) occurrences in a proof correspond to each other as follows:

\begin{itemize}
\item
  Every subformula (symbol) occurrence in a side formula of a premise
  directly corresponds to the same occurrence of that subformula
  (symbol) in the same side formula in the conclusion.
\item
  Every active formula of a premise directly correspond to the topmost
  subformula occurrence of the same formula in the principal formula of
  the conclusion.
\item
  Every subformula (symbol) occurrence in an active formula of a premise
  directly corresponds to the same occurrence of that subformula
  (symbol) in the corresponding subformula in the principal formula of
  the rule.
\item
  Two subformulas (symbols) correspond to each other by the transitive
  reflexive closure of direct correspondence.
\end{itemize}

\end{definition}

As by definition correspondence is reflexive and transitive, we get the
following definition for the equivalence classes of correspondence:

\begin{definition}[Family] A family is an equivalence class of \(\GassBox \)
occurrences which respect to correspondence. \end{definition}

For the following lemma and all the other results in this paper
concerning correspondence, we fix a proof tree \(\GassT  = (T, R)\) and
consider correspondence according to this complete proof tree even when
talking about subtrees \(T\uhp S\) of~\(\GassT \).

As usual, we have the subformula property.

\begin{lemma}[Subformula Property] Any subformula (symbol)
occurrence in a partial Gentzen style (pre-)proof \(T\uhp S\) in G3s corresponds to \emph{at least one} subformula (symbol)
occurrence of the sequent \(S\) of \(T\uhp S\).

Any subformula (symbol) occurrence in a complete Gentzen style
(pre-)proof~\(T\) in G3s corresponds to
\emph{exactly} one subformula (symbol) occurrence in the root sequent
\(S_r\) of \(T\). \end{lemma}

\subsection{Annotated S4}\label{annotated-s4}

As we have already seen, all symbol occurrences in a Gentzen style proof
can be divided in disjoint equivalence classes of corresponding symbol
occurrences. In this text we will be mainly concerned with the
equivalence classes of \(\GassBox \) occurrences, called families, and their
polarities as defined below. We will therefore define annotated formulas,
sequents and proof trees in this section which make the families and
polarities of \(\GassBox \) occurrences explicit in the notation and usable in
definitions.

\begin{definition}[Polarity] Assign \emph{positive} or
\emph{negative polarity} relative to \(A\) to all subformulas
occurrences \(B\) in \(A\) as follows:

\begin{itemize}
\item
  The only occurrence of \(A\) in \(A\) has positive polarity.
\item
  If an occurrence \(B \to  C\) in \(A\) already has a polarity, then the
  occurrence of \(C\) in \(B \to  C\) has the same polarity and the
  occurrence of \(B\) in \(B \to  C\) has the opposite polarity.
\item
  If an occurrence \(\GassBox B\) already has a polarity, then the occurrence of
  \(B\) in \(\GassBox B\) has the same polarity.
\end{itemize}

Similarly all occurrences of subformulas in a sequent \(\Gamma  \supset  \Delta \) get
assigned a \emph{polarity} as follows:

\begin{itemize}
\item
  An occurrence of a subformula \(B\) in a formula \(A\) in \(\Gamma \) has
  the opposite polarity relative to the sequent \(\Gamma  \supset  \Delta \) as the same
  occurrence \(B\) in the formula \(A\) has relative to \(A\).
\item
  An occurrence of a subformula \(B\) in a formula \(A\) in \(\Delta \) has
  the same polarity relative to the sequent \(\Gamma  \supset  \Delta \) as the same
  occurrence \(B\) in the formula \(A\) has relative to \(A\).
\end{itemize}

\end{definition}

This gives the subformulas of a sequent \(\Gamma  \supset  \Delta \) the same polarity as
they would have in the equivalent formula \(\bigwedge \Gamma  \to  \bigvee\Delta \). Also notice that
for the derived operators all subformulas have the same polarity, except
for \(\lnot \) which switches the polarity for its subformula.

The rules of S4 respect the polarities of the subformulas, so that all
corresponding occurrences of subformulas have the same polarity
throughout the proof. We therefore assign positive polarity to families
of positive occurrences and negative polarity to families of negative
occurrences. Moreover, positive families in a S4 proof which have
occurrences introduced by a \((\supset  \GassBox )\) rule are called principal positive
families or simply principal families. The remaining positive families
are called non-principal positive families.\footnote{This is the same
  terminology as used in Yu~\cite{yu2010}. In
  many papers principal families are called essential families following
  the original text~\cite{artemov2001}.}

The following definition of an annotated proof as well as the
definition of a realization function are
heavily inspired by Fittings use of explicit annotations in Fitting~\cite{fitting2009}. 
Other than Fitting, we allow
ourselves to treat symbols \(\boxplus _i\), \(\boxminus _i\) directly as mathematical
objects and define functions on them, instead of encoding the symbols as
natural numbers.

For the following definition, we use an arbitrary fixed enumeration for
all different classes of families. That is, we enumerate all principal
positive families as \(p_0, \ldots , p_{n_p}\), all non-principal positive
families as \(o_0, \ldots, o_{n_o}\) and all negative families as
\(n_0, \ldots, n_{n_n}\). Given a S4 proof \(T\) we then annotate the
formulas \(A\) in the proof in the following way:

\begin{definition}[Annotated Proof]

\(\an_T(A)\) is defined recursively on all occurrences of subformulas
\(A\) in a proof \(T\) as follows:

\begin{itemize}
\item
  If \(A\) is the occurrence of an atomic formula \(P\) or \(\bot \), then
  \(\an_T(A) := A\).
\item
  If \(A = A_0 \to  A_1\), then \(\an_T(A) := \an_T(A_0) \to  \an_T(A_1)\)
\item
  If \(A = \GassBox A_0\) and the \(\GassBox \) belongs to a principal positive family
  \(p_i\), then \(\an_T(A) := \boxplus _i \an_T(A_0)\).
\item
  If \(A = \GassBox A_0\) and the \(\GassBox \) belongs to a non-principal positive
  family \(o_i\), then \(\an_T(A) := \boxdot_i \an_T(A_0)\).
\item
  If \(A = \GassBox A_0\) and the \(\GassBox \) belongs to a negative family \(n_i\),
  then \(\an_T(A) := \boxminus _i \an_T(A_0)\).
\end{itemize}

\end{definition}

\subsection{Realization}\label{sec:realization}

LP and S4 are closely related and LP can be understood as an explicit
version of S4. The other way around, S4 can be seen as a version of LP
with proof details removed or forgotten. We will establish this close
relationship in this section formally by two main theorems translating
valid LP formulas into valid S4 formulas and vice versa. The former is
called forgetful projection, the latter is more complex and called
realization.

\begin{definition}[Forgetful Projection] The \emph{forgetful
projection} \(A\GassCircli  \) of a LP formula \(A\) is the following S4 formula:

\begin{itemize}
\item
  if \(A\) is atomic or \(\bot \), then \(A\GassCircli   := A\).
\item
  if \(A\) is the formula \(A_0 \to  A_1\) then \(A\GassCircli   := A_0\GassCircli   \to  A_1\GassCircli  \)
\item
  if \(A\) is the formula \(t{:}A_0\) then \(A\GassCircli   := \GassBox A_0\)
\end{itemize}

The definition is expanded to sets, multisets and sequents of LP
formulas in the natural way. \end{definition}

The forgetful projection maps LP theorems to S4 theorems.

\begin{theorem}

If\/ \(\LP \vdash  A\) then \(S4 \vdash  A\GassCircli  \). \end{theorem}

In the other direction, one can realize S4 formulas in LP by replacing
the \(\GassBox \) occurrences by explicit justification terms as defined below.
Of course most of this realizations will not transform a theorem of S4
into a theorem of LP. So the realization theorem will only assert the
existence of a specific realization producing a theorem of LP from a
theorem of S4. The constructive proof for the realization theorem also
provides us with an algorithm to generate one such realization. However,
that realization is not necessarily the only possible realization or the
simplest one.

\begin{definition}[Realization Function] A  \emph{realization function} \(r_T\) for a proof \(T\) is a
mapping from the set of different \(\GassBox \) symbols used in \(\an_T(T)\) to
arbitrary LP terms. \end{definition}

\begin{definition}[LP-Realization] By an \emph{LP-realization} of a
modal formula \(A\) we mean an assignment of proof polynomials to all
occurrences of the modality in \(A\) along with a constant specification
of all constants occurring in those proof polynomials. By \(A^r\) we
understand the image of \(A\) under a realization \(r\). \end{definition}

An LP-realization of \(A\) is fully determined by a realization function \(r_T\) relative to a proof
tree for \(\supset  A\) and a constant specification of all constants
occurring in \(r_T\) with 
\(A^r := r_T(\an_T(A))\).

If we read \(\GassBox A\) as \emph{there exists a proof for \(A\)} and \(t{:}A\) as
\emph{\(t\) is a proof for \(A\)}, this process seems immediately reasonable.
For the formula \(\lnot \GassBox A\), read as \emph{there is no proof of \(A\)}, and its
realization \(\lnot t{:}A\), read as \emph{\(t\) is not a proof of \(A\)}, on the
other hand, that process seems wrong at first. But justification logic
without any quantifications over proofs is still enough to capture the
meaning of \(\lnot \GassBox A\) by using Skolem's idea of replacing quantifiers with
functions. That is, we realize \(\lnot \GassBox A\) using an implicitly all
quantified justification variable \(\lnot x{:}A\). The same example
formulated without the derived connective \(\lnot \) is \(x{:}A \to  \bot \). That
formula can be read as function which produces a contradiction from a
given proof \(x\) for~\(A\).

This last interpretation also hints at the role of complex justification
terms using variables in a realization. They define functions from input
proofs named by the variables to output proofs for different formulas.
So a realization \[x{:}A \to  t(x){:}B\] of an S4 formula \(\GassBox A \to  \GassBox B\)
actually defines a function \(t(x)\) producing a proof for \(B\) from a
proof \(x\) for \(A\). This then is the Skolem style equivalent of the
quantified formula \( \exists (x) x{:}A \to  \exists (y) y{:}B\) which is the direct
reading of \(\GassBox A \to  \GassBox B\) (cf Artemov~\cite{artemov2008}). This discussion implies
that we should replace \(\GassBox \) with negative polarity with justification
variables, which leads to the following definition of a normal
realization:

\begin{definition}[Normal] A realization function is \emph{normal}
if all symbols for negative families and non-principal positive families
are mapped to distinct variables. A LP-realization is \emph{normal} if
the corresponding realization function is normal and the CS is
injective. \end{definition}

We are now ready to complete the connection between S4 and LP by the
following realization theorem giving a constructive way of producing the
necessary proof functions to realize a S4 theorem in LP:

\begin{theorem}[Realization]
\label{realization} If \(S4 \vdash  A\) then \(\LP \vdash  A^r\) for some normal
LP-realization \(r\). \end{theorem}

There are many proofs of the realization theorem available.
Artemov~\cite{artemov2001} already established it in his original paper on the Logic of Proofs.
Fitting~\cite{fitting2009} introduces a different proof-theoretic realization method.
Adapting his method to nested sequent systems yields a modular realization theorem that covers many modal logics~\cite{GoeKuz12APAL}. 
There is also a semantic proof of the realization theorem available~\cite{fitting2005} and Fitting~\cite{Fit13JLC} develops a  general realization method that uses the model existence theorem. 
A realization theorem can sometimes be obtained using a translation from one logic into another~\cite{BucKuzStu14Realizing,Fit11SynLib}.

\section{Prehistoric Relations in
G3s}\label{prehistoric-relations-in-g3s}

\sectionmark{Prehist. Rel. in G3s}

\subsection{Self-referentiality}\label{self-referentiality}

As already mentioned in the introduction, the formulation of LP allows
for terms \(t\) to justify formulas \(A(t)\) about themselves. We will
see that such self-referential justification terms are not only
possible, but actually unavoidable for realizing S4 even at the basic
level of justification constants. That is to realize all S4 theorems in
LP, we need self-referential constant specifications defined as follows:

\begin{definition}[Self-Referential Constant Specification]
A constant specification CS is 
\begin{itemize}
\item
  \emph{directly self-referential} if
  there is a constant \(c\) such that \(c{:}A(c) \in  \CS\).
\item
  \emph{self-referential} if there is a
  subset \(A \subseteq  \CS\) such that
  \[A := \{c_0{:}A(c_1), \ldots, c_{n-1}{:}A(c_0)\}.\]
\end{itemize}

\end{definition}

A constant specification which is not directly self-referential is
denoted by~\(\CS^*\). Similarly a constant specification which is not
self-referential at all is denoted by \(\CS^\odot \). So \(\CS^*\) and
\(\CS^\odot \) stand for a class of constant specifications and not a single
specific one. Following Yu~\cite{yu2010}, we
 use the notation \(\LP(\CS^\odot ) \vdash  A\) if there exists any
non-self-referential constant specification CS such that
\(\LP(\CS) \vdash  A\). There does exist a single maximal constant
specification \(\CS_{nds}\) that is not directly self-referential and
for any theorem \(A\) we have \(\LP(\CS^*) \vdash  A\) iff
\(\LP(\CS_{nds}) \vdash  A\).

Given that any S4 theorem is realizable in LP with some constant
specification, we can carry over the definition of self-referentiality
to S4 with the following definition:

\begin{definition}[Self-Referential Theorem] An S4 theorem \(A\) is
(directly) self-referential iff for any LP-realization \(A^r\) we have that
\(\LP(\CS^\odot ) \slashed{\vdash } A^r\) (respectively
\(\LP(\CS^*) \slashed{\vdash } A^r\)). \end{definition}

Expanding on a first result for S4 in Brezhnev and Kuznets~\cite{brezhnev2006}, Kuznets~\cite{kuznets2010} explores the topic of
self-referentiality on the level of individual modal logics and their
justification counterparts. He gives theorems for the modal logics S4,
D4, T, and K4 which can only be realized in their justification logic
counterpart using directly self-referential constant specifications,
i.e.~directly self-referential theorems by the above definition. So for
S4 in particular, Kuznets gives the theorem \(\lnot \GassBox \lnot (S \to  \GassBox S)\) and shows
that it is directly self-referential.

We will not reproduce this result but use the logically equivalent
formula \(\lnot \GassBox (P \land  \lnot \GassBox P)\) as an example for a self-referential S4 theorem.
Notice that it does not directly follow from the above theorem that
\(\lnot \GassBox (P \land  \lnot \GassBox P)\) can only be realized with a self-referential constant
specification, as justification terms do not necessary apply to
logically equivalent formulas.
Still it should be
fairly straightforward to show that \(\lnot \GassBox (P \land  \lnot \GassBox P)\) is self-referential
by translating justification terms for the outer \(\GassBox \) occurrences in the formulas
\(\lnot \GassBox (P \land  \lnot \GassBox P)\) and \(\lnot \GassBox \lnot (S \to  \GassBox S)\) using the logical equivalence of
\(P \land  \lnot \GassBox P\) and \(\lnot (S \to  \GassBox S)\).

Looking at the G3s proof for \(\lnot \GassBox (P \land  \lnot \GassBox P)\) and a realization of that
proof in figure \ref{proofs}, we can see why a self referential term
like \(t\) for the propositional tautology \(P \land  \lnot t\cdot x{:}P \to  P\) is
necessary. In order to prove \(\lnot \GassBox (P \land  \lnot \GassBox P)\) one needs to disprove
\(P \land  \lnot \GassBox P\) at some point which means one has to prove \(\GassBox P\). The only
way to prove \(\GassBox P\) is using \(\GassBox (P \land  \lnot \GassBox P)\) as an assumption on the
left. This leads to the situation that the proof introduces \(\GassBox \) by a
\((\supset  \GassBox )\) rule where the same family already occurs on the left. 
In the following, we will see that such a situation
is actually necessary for the self-referentiality of any S4 formula.

\begin{figure} \caption{proof for $\lnot \GassBox (P \land  \lnot \GassBox P)$ and its realization} \label{proofs}
\begin{longtable}{cc}
\AXC{$P, \lnot \GassBox P, \GassBox (P\land \lnot \GassBox P) \supset  P$}
\RightLabel{$(\land  \supset )$}
\UIC{$P \land  \lnot \GassBox P, \GassBox (P\land \lnot \GassBox P) \supset  P$}
\RightLabel{$(\GassBox  \supset )$}
\UIC{$\GassBox (P\land \lnot \GassBox P) \supset  P$}
\RightLabel{$(\supset  \GassBox )$}
\UIC{$P, \GassBox (P\land \lnot \GassBox P) \supset  \GassBox P$}
\RightLabel{$(\lnot  \supset )$}
\UIC{$P, \lnot \GassBox P, \GassBox (P\land \lnot \GassBox P) \supset $}
\RightLabel{$(\land  \supset )$}
\UIC{$P \land  \lnot \GassBox P, \GassBox (P\land \lnot \GassBox P) \supset $}
\RightLabel{$(\GassBox  \supset )$}
\UIC{$\GassBox (P \land  \lnot \GassBox P) \supset $}
\RightLabel{$(\supset  \lnot )$}
\UIC{$\supset  \lnot \GassBox (P \land  \lnot \GassBox P)$}
\DP

&

\AXC{$P, \lnot t\cdot x{:}P, x{:}(P \land  \lnot t\cdot x{:}P) \supset  P$}

\UIC{$P \land  \lnot t\cdot x{:}P, x{:}(P \land  \lnot t\cdot x{:}P) \supset  P$}

\UIC{$x{:}(P \land  \lnot t\cdot x{:}P) \supset  P$}

\UIC{$P, x{:}(P \land  \lnot t\cdot x{:}P) \supset  t\cdot x{:}P$}

\UIC{$P, \lnot t\cdot x{:}P, x{:}(P \land  \lnot t\cdot x{:}P) \supset $}

\UIC{$P \land  \lnot t\cdot x{:}P, x{:}(P \land  \lnot t\cdot x{:}P) \supset $}

\UIC{$x{:}(P \land  \lnot t\cdot x{:}P) \supset $}

\UIC{$\supset  \lnot x{:}(P \land  \lnot t\cdot x{:}P)$}
\DP
\end{longtable}
\end{figure}

\subsection{Prehistoric Relations}\label{prehistoric-relations}

In his paper ``Prehistoric Phenomena and Self-referentiality''~\cite{yu2010}, Yu gives a formal definition for the
situation described in the last section, which he calls a prehistoric
loop. In the later paper~\cite{yu2017}, Yu
adopts the proper graph theoretic term cycle as we do here. Beside that
change we will reproduce his definitions of prehistoric relation,
prehistoric cycle as well as some basic lemmas about this new notions
exactly as they were presented in the original paper.

To work with the \((\supset  \GassBox )\) rules introducing occurrences of principal
families in a G3s proof, we will use the following notation: 
we enumerate all \((\supset  \GassBox )\) rules
introducing an occurrence of the principal family \(p_i\) as
\(R_{i,0}, \ldots R_{i,l_i-1}\) and use \(I_{i,0}, \ldots I_{i,l_i-1}\) to
denote the premises of those rules and \(O_{i,0}, \ldots O_{i,l_i-1}\) to
denote their conclusions, see
  Yu~\cite{yu2010}

\begin{definition}[History] In a root-leaf path \(S\) of the form
\(S_rR^*O_{i,j}RI_{i,j}R^*S\) in a G3s-proof \(T\), the path
\(S_rR^*O_{i,j}\) is called a \emph{history} of the family \(p_i\) in
the root-leaf path \(S\). The path \(I_{i,j}R^*S\) is called a
\emph{pre-history} of \(p_i\) in the root-leaf path \(S\).
 \end{definition}

So intuitively every \((\supset  \GassBox )\) rule divides a root-leaf path of the
proof tree into two parts. The first part from the root of the tree to
the conclusion of the \((\supset  \GassBox )\) rule of sequents having a copy of that
\(\GassBox \) symbol, i.e.~the history of that \(\GassBox \) symbol from its formation
up to the root sequent. And the second part which predates the formation
of that \(\GassBox \) symbol, i.e.~all sequents from the leaf up to the premise
of that \((\supset  \GassBox )\) rule, which do not have a copy of that symbol. The
informal notion of ``having a copy of that symbol'' is not the same as
correspondence, as it is not transitively closed. It is possible to have
corresponding \(\boxplus _i\) occurrences of a family \(p_i\) in a prehistory of
that same family. 
The proof in figure~\ref{proofs} of our example theorem
exhibits this case.

As we are especially interested in these cases, that is occurrences of
principal families in prehistoric periods, the following definition and
lemma give that concept a precise meaning and notation:

\begin{definition}[Prehistoric Relation] \label{local1} For any
principal positive families \(p_i\) and \(p_h\) and any root-leaf path
\(S\) of the form \(S_rR^*O_{i,j}RI_{i,j}R\ast S\) in a S4 proof
\(\GassT  = (T, R)\):

\begin{enumerate}
\def\labelenumi{(\arabic{enumi})}
\item
  If \(\an_T(I_{i,j})\) has the form
  \(\boxminus _{k_0}B_{k_0}, \ldots, \boxminus _{k}B_k(\boxplus _h C), \ldots, \boxminus _{k_q}B_{k_q} \supset  A\),
  then \(p_h\) is a \emph{left prehistoric family} of \(p_i\) in \(S\)
  with notation \(h \prec ^S_L i\).
\item
  If \(\an_T(I_{i,j})\) has the form
  \(\boxminus _{k_0} B_{k_0}, \ldots, \boxminus _{k_q}B_{k_q} \supset  A(\boxplus _h C)\) then \(p_h\) is a
  \emph{right prehistoric family} of \(p_i\) in \(S\) with notation
  \(h \prec ^S_R i\).
\item
  The relation of \emph{prehistoric family} in \(S\) is defined by:
  \(\prec ^S := \prec ^S_L \cup  \prec ^S_R\). The relation of \emph{(left, right)
  prehistoric family} in \(T\) is defined by:
  \[
  \prec _L := \bigcup \{\prec ^S_L \ \mid \text{$S$ is a leaf}\}
  \qquad
  \prec _R := \bigcup \{\prec ^S_R \ \mid \text{$S$ is a leaf}\}
  \]
  and \(\prec  := \prec _L \cup  \prec _R\).
\end{enumerate}

\end{definition}

Even though both definitions so far use the notion of a prehistory, they
do not directly refer to each other. But the following lemma provides
the missing connection between these two definitions and therefore
explains the common terminology:

\begin{lemma} \label{global} There is an occurrence of \(\boxplus _h\) in a
pre-history of \(p_i\) in the root-leaf path \(S\) iff \(h \prec ^S i\).
\end{lemma}

\begin{proof}

(\GassRight ): \(\boxplus _h\) occurs in a sequent \(S'\) in a pre-history of \(p_i\) in
the root-leaf path \(S\), so \(S\) has the form
\(S_rR^*O_{i,j}RI_{i,j}R^*S'R^*S\) for some \(j < l_i\). By the
subformula property, there is an occurrence of \(\boxplus _h\) in \(I_{i,j}\) as
\(S'\) is part of \(T\uhp I_{i,j}\). If this occurrence is on the left we
have \(h \prec ^S_L i\), if it is on right we have \(h \prec ^S_R i\). In both
cases \(h \prec ^S i\) holds.

(\GassLeft ): By definition there is a \(I_{i,j}\) in \(S\), where \(\boxplus _h\) occurs
either on the left (for \(h \prec ^S_L i\)) or on the right (for
\(h \prec ^S_R i\)). \(I_{i,j}\) is part of the pre-history of \(R_{i,j}\) in
\(S\). \end{proof}

Having introduced the concepts of prehistoric periods and prehistoric
relations, we are now ready to define the concept of prehistoric cycles
used in Yu's theorem:

\begin{definition}[Prehistoric Cycle] In a G3s-proof \(T\), the
ordered list of principal positive families \(p_{i_0},\ldots, p_{i_{n-1}}\)
with length \(n\) is called a \emph{prehistoric cycle} or \emph{left
prehistoric cycle} respectively, if we have:
\(i_0 \prec  i_2 \prec  \cdots \prec  i_{n-1} \prec  i_0\) or
\(i_0 \prec _L i_2 \prec _L \cdots \prec _L i_{n-1} \prec _L i_0\). \end{definition}

In our example formula, we have a prehistoric cycle consisting of a
single principal family which has a left prehistoric relation to itself.
The following lemma shows that  if a proof has a prehistoric cycle, then it
also has a left prehistoric cycle:

\begin{lemma}

\(T\) has a prehistoric cycle iff \(T\) has a left prehistoric cycle.
\end{lemma}

Finally, Yu~\cite{yu2010} showed that left prehistoric cycles are necessary for self-referentiality

\begin{theorem}[Necessity of Left Prehistoric Cycle for
Self-referentiality] If a S4-theorem \(A\) has a
left-prehistoric-cycle-free G3s-proof, then there is a LP-formula \(B\)
s.t. \(B^\circ = A\) and \(\LP(\CS^\odot ) \vdash  B\) \end{theorem}

\section{Prehistoric Relations in
G3lp}\label{prehistoric-relations-in-g3lp}

\sectionmark{Prehist. Rel. in G3lp}

\subsection{Cut Rules}\label{cut-rules}

In this section we will prepare our discussion of prehistoric relations
for LP, by first expanding the notion of families and prehistoric
relations to the systems G3s + (Cut) and G3s + (\GassBox Cut) using cut rules.
The (context sharing) cut rule has the following definition:

\begin{definition}[(Cut) Rule]

\AXC{$\Gamma  \supset  \Delta , A$}
\AXC{$A, \Gamma  \supset  \Delta $}
\RightLabel{(Cut)}
\BIC{$\Gamma  \supset  \Delta $}
\DP

\end{definition}

It is necessary to expand the definition of correspondence (def.
\ref{corr}) to (Cut) rules as follows:

\begin{definition}[Correspondence for (Cut)]
%
  The active formulas (and their symbols) in the premises of a (Cut)
  rule correspond to each other.
%
\end{definition}

The classification and annotations for families of \(\GassBox \) do not carry
over to G3s + (Cut), as the (Cut) rule uses the cut formula in different
polarities for the two premises. We therefore will consider \emph{all}
\(\GassBox \) families for prehistoric relations in G3s + (Cut) proofs. This
leads to the following expanded definition of prehistoric relation:

\begin{definition}[Local Prehistoric Relation in G3s + (Cut)]
\label{local2} A family \(\GassBox _i\) has a \emph{prehistoric relation} to
another family \(\GassBox _j\), in notation \(i \prec  j\), if there is a \((\supset  \GassBox )\)
rule introducing an occurrence of \(\GassBox _j\) with premise \(S\), such that
there is an occurrence of \(\GassBox _i\) in \(S\). \end{definition}

Notice that there can be prehistoric relations with \(\GassBox \) families which
locally have negative polarity, as the family could be part of a cut
formula and therefore also occur with positive polarity in the other
branch of the cut. On the other hand, adding prehistoric relations with
negative families in a cut free G3s proof does not introduce prehistoric
cycles, as in G3s a negative family is never introduced by a \((\supset  \GassBox )\)
rule and therefore has no prehistoric families itself. In G3s + (Cut)
proofs, the subformula property and therefore also lemma \ref{global} no
longer hold. That means we can have an occurrence of a family \(\GassBox \) as
part of a cut formula in the \emph{global} prehistory of a \((\supset  \GassBox )\)
rule, which by the \emph{local} definition \label{defcut} is not a local
prehistoric family.

To handle terms \(s\cdot t\) in the next section an additional rule for modus
ponens under \(\GassBox \) is necessary. We therefore introduce here the new rule
(\GassBox Cut) as follows:

\begin{definition}[(\GassBox Cut) Rule]
\[
\AXC{$\Gamma  \supset  \Delta , \GassBox A, \GassBox B$}
\AXC{$\Gamma  \supset  \Delta , \GassBox (A \to  B), \GassBox B$}
\RightLabel{(\GassBox Cut)}
\BIC{$\Gamma  \supset  \Delta , \GassBox B$}
\DP
\]
\end{definition}

Again it is also necessary to expand the definition of correspondence
(def.~\ref{corr}) for this rule:

\begin{definition}[Correspondence for (\GassBox Cut)] \label{boxcutcorr}
\hfill
\begin{itemize}
\item
  The topmost \(\GassBox \) occurrence in the active formulas and the principal
  formula correspond to each other.
\item
  The subformulas \(A\) in the active formulas of the premises
  correspond to each other.
\item
  The subformulas \(B\) correspond to each other.
\end{itemize}

\end{definition}

Notice that with this expansion \(\GassBox \) occurrences of the same family no
longer are always part of the same subformula \(\GassBox C\).
Also similar to the (Cut) rule,
correspondence is expanded to relate negative and positive occurrences
of \(\GassBox \) symbols as \(A\) is used with different polarities in the two
premises.

With the following lemmas and theorems we will establish a constructive
proof for
\[\Gs + (\GassBox \Cut) \vdash  \Gamma  \supset  \Delta  \GassRight  \Gs + (\Cut) \vdash  \Gamma  \supset  \Delta  \GassRight  \Gs \vdash  \Gamma  \supset  \Delta.\] 
Moreover
there will be corollaries showing that the constructions do not
introduce prehistoric cycles by the new definition \ref{local2}. As all
prehistoric relations by the first definition \ref{local1} are included
in the new definition, the final proof in G3s will be
prehistoric-cycle-free by any definition if the original proof in \mbox{G3s +(\GassBox Cut)} 
was prehistoric-cycle-free by the new definition.


We need the following standard result, which we mention without proof.
\begin{lemma}\label{l:contraction:1}
Weaking, inversion, and contraction are admissible in G3s.
Moreover,  for any annotation the
constructed proofs do not introduce any new prehistoric relations.
\end{lemma}

It is important to note, that the second part of the lemma is not
restricted to the annotations \(\an_T\) of the proofs \(\GassT  = (T, R)\)
given by the premise of the lemma but still hold for arbitrary
annotations \(\an\). That means there is no implicit assumption that the
families have only a single occurrence in the root sequents used in the
lemma or theorem and the results can also be used in subtrees \(T\uhp S\)
together with an annotation \(\an_T\) for the complete tree.

In the case for weakening, the second part of the lemma also follows from the fact 2.8 in Yu~\cite{yu2017}. 
There, Yu looks at
prehistoric relations locally, i.e.~taking only correspondence up to the
current sequent in consideration. That means the graph of prehistoric
relations has to be updated going up the proof tree as new rules add new
correspondences and therefore unify vertices in the prehistoric
relations graph which were still separate in the premise. To work with
such changing graphs, Yu introduces the notion of isolated families. He
shows that all \(\GassBox \) occurrences introduced by weakening are isolated.
That means they have no prehistoric relations themselves, which globally
means that they can not add any prehistoric relations from adding
correspondences later in the proof. This is exactly what the second part asserts for weakening.



\begin{theorem}[Cut Elimination for G3s] \label{cut} If
\[\Gs \vdash  \Gamma  \supset  \Delta , A \text{ and } \Gs \vdash  A, \Gamma  \supset  \Delta, \] 
then \(\Gs \vdash  \Gamma  \supset  \Delta \).
\end{theorem}

\begin{proof}

By a simultaneous induction over the depths of the proof trees \(\GassT _L\)
for \(\Gamma  \supset  \Delta , A\) and \(\GassT _R\) for \(A, \Gamma  \supset  \Delta \) as well as the rank of
\(A\).
For us, the only interesting case is:

\(A\) is a side formula in the last rule of \(\GassT _R\), which is
a \((\supset  \GassBox )\) rule and a principal formula in the last rule of \(\GassT _L\).
Then \(A\) has the form \(\GassBox A_0\) as it is a side formula of a \((\supset  \GassBox )\)
on the right. So the last rule of \(\GassT _L\) is also a \((\supset  \GassBox )\) rule and
the proof has the following form:

\AXC{$\GassT _L$} \noLine \UIC{$\GassBox \Gamma _L \supset  A_0$} \RightLabel{$(\supset  \GassBox )$}
\UIC{$\Gamma '_L, \GassBox \Gamma _L \supset  \Delta ', \GassBox B, \GassBox A_0$}

\AXC{$\GassT _R$} \noLine \UIC{$\GassBox A_0, \GassBox \Gamma _R \supset  B$} \RightLabel{$(\supset  \GassBox )$}
\UIC{$\Gamma '_R, \GassBox A_0, \GassBox \Gamma _R \supset  \Delta ', \GassBox B$}

\RightLabel{(Cut)}
\BIC{$\Gamma  \supset  \Delta ', \GassBox B$}
\DP

where \(\Delta  = \Delta ', \GassBox B\) and \(\Gamma  = \Gamma '_L, \GassBox \Gamma _L = \Gamma '_R, \GassBox \Gamma _R\).

The cut can be moved up on the right using weakening as follows:

\AXC{$\GassT _L$} \noLine \UIC{$\GassBox \Gamma _L \supset  A_0$} \RightLabel{$(\supset  \GassBox )$}
\UIC{$\GassBox \Gamma _R, \GassBox \Gamma _L \supset  B, \GassBox A_0$}

\AXC{$\GassT '_R$} \noLine \UIC{$\GassBox A_0, \GassBox \Gamma _R, \GassBox \Gamma _L \supset  B$}

\RightLabel{(Cut)}
\BIC{$\GassBox \Gamma _R, \GassBox \Gamma _L \supset  B$}

\RightLabel{$(\supset  \GassBox )$}
\UIC{$\Gamma , \GassBox \Gamma _R, \GassBox \Gamma _L \supset  \Delta ', \GassBox B$}
\DP

By the induction hypothesis and a contraction we get the required proof
for \(\Gamma  \supset  \Delta \) as \(\GassBox \Gamma _L \subseteq  \Gamma \) and \(\GassBox \Gamma _R \subseteq  \Gamma \).
\end{proof}

\begin{corollary} \label{cutprehist} For any annotation \(\an\) the
constructed proof for \(\Gamma  \supset  \Delta \) only introduces new prehistoric
relations \(i \prec  j\) between families \(\GassBox _i\) and \(\GassBox _j\) occurring in
\(\Gamma  \supset  \Delta \) where there exists a family \(\GassBox _k\) in \(A\) such that
\(i \prec  k \prec  j\) in the original proof. \end{corollary}

\begin{proof}

The only place where new prehistoric relations get introduced is by
the new \((\supset  \GassBox )\) in the case shown above. 
All prehistoric relations from \(\GassBox \Gamma _R\)
are already present from the \((\supset  \GassBox )\) rule on the right in the original
proof. So only prehistoric relations from \(\GassBox \Gamma _L\) are new. For all
families \(\GassBox _i\) in \(\GassBox \Gamma _L\) we have \(i \prec  k\) for the family \(\GassBox _k\) in
the cut formula introduced by the \((\supset  \GassBox )\) rule on the left. Moreover
\(k \prec  j\) for the same family because of the occurrence of \(\GassBox A_0\) on
the right. \end{proof}

\begin{corollary} \label{cutcycle} For any annotation \(\an\) the
constructed proof for \(\Gamma  \supset  \Delta \) does not introduce prehistoric cycles.
\end{corollary}

\begin{proof}

Assume for contradiction that there exists a prehistoric cycle
\[i_0 \prec  \cdots \prec  i_{n-1} \prec  i_0\] in the new proof. By the previous lemma
for any prehistoric relation \[i_k \prec  i_{k+1 \mod n}\] in the cycle
either \(i_k \prec  i_{k+1 \mod n}\) in the original proof or there is a
family \(i'_k\) in the cut formula such that
\(i_k \prec  i'_k \prec  i_{k+1 \mod n}\) in the original proof. Therefore we also
have a prehistoric cycle in the original proof. \end{proof}

\begin{theorem}[(\GassBox Cut) Elimination] \label{boxcut} If
\[\Gs \vdash  \Gamma  \supset  \Delta , \GassBox A, \GassBox B \text{ and } \Gs \vdash  \Gamma  \supset  \Delta , \GassBox (A \to  B), \GassBox B,\] then
\(\Gs \vdash  \Gamma  \supset  \Delta , \GassBox B\) \end{theorem}

\begin{proof}

By a structural induction over the proof trees \(\GassT _L\) for
\(\Gamma  \supset  \Delta , \GassBox A, \GassBox B\) and \(\GassT _R\) for \(\Gamma  \supset  \Delta , \GassBox (A \to  B), \GassBox B\).

1.~case: \(\GassBox (A \to  B)\) or \(\GassBox A\) is a weakening formula of the last rule.
Then removing them from that proof gives the required proof. This
includes the case when \(\GassBox B\) is the principal formula of the last rule
of either proof, as then the last rule is \((\supset  \GassBox )\) which has no side
formulas on the right.

2.~case: \(\GassBox (A \to  B)\) or \(\GassBox A\) is a side formula of the last rule. Then
also \(\GassBox B\) is a side formula of that rule. Use the induction hypothesis
on the premises of that rule with the other proof and append the same
rule.

3.~case: \(\GassBox (A \to  B)\) and \(\GassBox A\) are the principal formula of the last
rule. Then the last rules have the following form:

\AXC{$\GassT _L$} \noLine \UIC{$\GassBox \Gamma _L \supset  A$} \RightLabel{$(\supset  \GassBox )$}
\UIC{$\Gamma _L', \GassBox \Gamma _L  \supset  \Delta , \GassBox A, \GassBox B$}

\AXC{$\GassT _R$} \noLine \UIC{$\GassBox \Gamma _R \supset  A \to  B$} \RightLabel{$(\supset  \GassBox )$}
\UIC{$\Gamma '_R, \GassBox \Gamma _R  \supset  \Delta , \GassBox (A \to  B), \GassBox B$}

\RightLabel{(\GassBox Cut)}
\BIC{$\Gamma  \supset  \Delta , \GassBox B$}
\DP

where \(\Delta  = \Delta ', \GassBox B\) and \(\Gamma  = \Gamma '_L, \GassBox \Gamma _L = \Gamma '_R, \GassBox \Gamma _R\).

By inversion for \((\supset  \to )\) we get a proof \(\GassT '_R\) for \(A, \GassBox \Gamma _R \supset  B\)
from the first premise \(\GassBox \Gamma _R \supset  A \to  B\). Using weakening and a normal
cut on the formula \(A\) we get the following proof:

\AXC{$\GassT '_L$} \noLine \UIC{$\GassBox \Gamma _L, \GassBox \Gamma _R \supset  A$} \AXC{$\GassT ''_R$} \noLine
\UIC{$A, \GassBox \Gamma _L, \GassBox \Gamma _R  \supset  B$} \RightLabel{(Cut)} \BIC{$\GassBox \Gamma _L, \GassBox \Gamma _R \supset  B$}
\RightLabel{$(\supset  \GassBox )$} \UIC{$\Gamma , \GassBox \Gamma _L, \GassBox \Gamma _R \supset  \Delta , \GassBox B$} \DP

By contraction and a cut elimination we get the required G3s proof for
\(\Gamma  \supset  \Delta , \GassBox B\) as \(\GassBox \Gamma _L \subseteq  \Gamma \) and \(\GassBox \Gamma _R \subseteq  \Gamma \). \end{proof}

\begin{corollary} \label{boxcutcycle} For any annotation \(\an\) the
constructed proof for \(\Gamma  \supset  \Delta , \GassBox B\) does not introduce prehistoric
cycles. \end{corollary}

\begin{proof}

Removing weakening or side formulas \(\GassBox (A\to B)\) or \(\GassBox A\) as in case 1
and~2 does not introduce new prehistoric relations.

Any prehistoric relation because of the new \((\supset  \GassBox )\) rule in case 3
already exists in the original proof, as every \(\GassBox \) occurrence in
\(\GassBox \Gamma _L\) or \(\GassBox \Gamma _R\) also occurs in one of the two \((\supset  \GassBox )\) rules in
the original proof, which both introduce a \(\GassBox\)  of the same family as
\(\GassBox B\) by the definition of correspondence for (\GassBox Cut). 

So the new proof with (\GassBox Cut) rules replaced by (Cut) rules does not
introduce new prehistoric relations and therefore also no new
prehistoric cycles. By corollary \ref{cutcycle}, the cut elimination to
get a G3s proof does not introduce prehistoric cycles. \end{proof}

\begin{definition}

The cycle-free fragment of a system \(Y\), denoted by \(Y^\otimes \), is the
collection of all sequents that each have a prehistoric-cycle-free
\(Y\)-proof.
\end{definition}

\begin{theorem}

The cycle-free fragments of G3s + (\GassBox Cut), G3s + (Cut) and G3s are
identical. \end{theorem}

\begin{proof}

A prehistoric-cycle-free proof in G3s by the original definition~\ref{local1} is also prehistoric-cycle-free by the new definition~\ref{local2} as a negative family can not have any prehistoric families
itself in a G3s-proof . So any sequent \[\Gamma  \subset  \Delta  \in  G3s^\otimes \] is trivially
also provable prehistoric-cycle-free in G3s + (Cut) and G3s + (\GassBox Cut) and
we have \(\Gs^\otimes  \subseteq  (\Gs + (\GassBox \Cut))^\otimes \) and \(\Gs^\otimes  \subseteq  (\Gs + (\Cut))^\otimes \).
Moreover \((\Gs + (\Cut))^\otimes  \subseteq  \Gs^\otimes \) by corollary \ref{cutcycle} and
\[(\Gs + (\GassBox \Cut))^\otimes  \subseteq  (\Gs + (\Cut))^\otimes  \subseteq  \Gs^\otimes \] by corollary~\ref{boxcutcycle}. All together we get
\[\Gs^\otimes  = (\Gs + (\Cut))^\otimes  = (\Gs + (\GassBox \Cut))^\otimes \qedhere\] \end{proof}

Yu~\cite[th.~2.21]{yu2017} shows
that non-self-referentiality is not normal in T, K4, and S4. The results
in this section hint at an explanation for this fact for S4 and at the
possibility to still use modus ponens with further restrictions in the
non-self-referential subset of S4. Namely, to consider the global
aspects of self-referentiality coming from correspondence of
occurrences, it is necessary when combining two proofs, that the two
proofs together with the correct correspondences added are
prehistoric-cycle-free. So we can only use modus ponens on two
non-self-referential S4 theorems \(A\) and \(A \to  B\) if there are proofs
of \(A\) and \(A \to  B\) such that the prehistoric relations of these
proofs combined, together with identifying the occurrences of \(A\) in
both proofs, are prehistoric-cycle-free. In that case we get a
prehistoric-cycle-free G3s proof for \(B\) using cut elimination and
corollary \ref{cutprehist}, which shows that \(B\) is also
non-self-referential.

\subsection{A sufficient condition for self-referentiality}\label{g3lp}

Cornelia Pulver~\cite{pulver2010}
introduces the system LPG3 by expanding G3c with rules for the build up
of justification terms as well as the new axioms (Axc) and (Axt). To
ensure that the contraction lemma holds, all rules have to be invertible,
which is the
reason why contracting variants of all the justification rules are used
for LPG3. Our variant G3lp will use the same rules to build up terms,
but replace the axioms with rules \((\supset  :)_c\) and \((\supset  :)_t\) to keep
the prehistoric relations of the proof intact. As there is a proof for
\(\supset  A\) for any axiom \(A\) and also for \(A \supset  A\) for any formula
\(A\), these two rules are equivalent to the two axioms and invertible.

As we already did with G3s, we will use the full system with all
classical operators for examples, but only the minimal subset with \(\to \)
and \(\bot \) for proofs. So these two systems use the classical rules from
G3s as well as the new LP
rules in figure \ref{G3lprules}.

\renewcommand{\arraystretch}{3}
\begin{figure} \caption{G3lp} \label{G3lprules}
\begin{longtable}{cc}

\AXC{$\supset  A$}
\RightLabel{$(\supset  :)_c$ ($A$ an axiom of LP)}
\UIC{$\Gamma  \supset  \Delta , c{:}A$}
\DP

&

\AXC{$t{:}A \supset  A$}
\RightLabel{$(\supset  :)_t$}
\UIC{$t{:}A, \Gamma  \supset  \Delta , t{:}A$}
\DP

\\

\RightLabel{$({:} \supset )$}
\AXC{$A, t{:}A, \Gamma  \supset  \Delta $}
\UIC{$t{:}A, \Gamma  \supset  \Delta $}
\DP

&

\RightLabel{$(\supset  !)$}
\AXC{$\Gamma  \supset  \Delta , t{:}A, !t{:}t{:}A$}
\UIC{$\Gamma  \supset  \Delta , !t{:}t{:}A$}
\DP

\\

\RightLabel{$(\supset  +)$}
\AXC{$\Gamma  \supset  \Delta , s{:}A, t{:}A, (s+t){:}A$}
\UIC{$\Gamma  \supset  \Delta , (s+t){:}A$}
\DP

&

\RightLabel{$(\supset  \cdot )$}
\AXC{$\Gamma  \supset  \Delta , s{:}(A \to  B), s\cdot t{:}B$}
\AXC{$\Gamma  \supset  \Delta , t{:}A, s\cdot t{:}B$}
\BIC{$\Gamma  \supset  \Delta , s\cdot t{:}B$}
\DP

\end{longtable}
\end{figure}

This system is adequate for the logic of proofs LP as shown in corollary
4.37 in Pulver~\cite{pulver2010}. It also
allows for weakening, contraction and inversion. By corollary 4.36 in
the same paper, G3lp without the \((\supset  :)_c\) rule is equivalent to
\(\LP_0\). Neither Pulver~\cite{pulver2010} nor
Artemov~\cite{artemov2001} define Gentzen
systems for a restricted logic of proofs LP(CS), perhaps because it
seems obvious that restricting whatever rule is used for introducing
proof constants to CS gives a Gentzen system for LP(CS).

To work with prehistoric relations in G3lp proofs we need the following
new or adapted definitions:

\begin{definition}[Subformula] The set of subformulas \(\sub(A)\) of
a LP formula \(A\) is inductively defined as follows:

\begin{enumerate}
\def\labelenumi{\arabic{enumi}.}
\item
  \(\sub(P) = \{P\}\) for any atomic formula \(P\)
\item
  \(\sub(\bot ) = \{\bot \}\)
\item
  \(\sub(A_0 \to  A_1) = \sub(A_0) \cup  \sub(A_1) \cup  \{A_0 \to  A_1\}\)
\item
  \(\sub(s+t{:}A_0) = \sub(A_0) \cup  \{s{:}A_0, t{:}A_0, s+t{:}A_0\}\)
\item
  \(\sub(t{:}A_0) = \sub(A_0) \cup  \{t{:}A_0\}\)
\end{enumerate}

\end{definition}

\begin{definition}[Subterm] The set of subterms \(\sub(t)\) of a LP
justification term \(t\) is inductively defined as follows:

\begin{enumerate}
\def\labelenumi{\arabic{enumi}.}
\item
  \(\sub(x) = \{x\}\) for any variable \(x\)
\item
  \(\sub(c) = \{c\}\) for any constant \(c\)
\item
  \(\sub(!t) = \sub(t) \cup  \{!t\}\)
\item
  \(\sub(s+t) = \sub(s) \cup  \sub(t) \cup  \{s + t\}\)
\item
  \(\sub(s\cdot t) = \sub(s) \cup  \sub(t) \cup  \{s\cdot t\}\)
\end{enumerate}

The set of subterms \(\sub(A)\) of a LP formula \(A\) is the union of
all sets of subterms for all terms occurring in \(A\). \end{definition}

We use the symbol \(\sub\) for all definitions of subterms and
subformulas, as it will be clear from context which of the definitions
is meant. Notice that by this definition \(s{:}A\) is a subformula of
\(s+t{:}A\).

We expand the definition of correspondence to G3lp
proofs as follows:

\begin{definition}[Correspondence in G3lp] All topmost terms in
active or principal formulas in the rules \((\supset  \cdot )\), \((\supset  +)\) \((\supset  !)\)
and \(({:} \supset )\) correspond to each other. \end{definition}

Notice that in the \((\supset  !)\) rule, the topmost term \(t\) in the
contraction formula therefore corresponds to the topmost proof term
\(!t\) in the principal formula. The term \(t\) of the other active
formula \(!t{:}t{:}A\) on the other hand corresponds to the same term
\(t\) in the principal formula.

By this definition, families of terms in G3lp consist not of occurrences
of a single term \(t\) but of occurrences of subterms \(s\) of a top
level term \(t\). We will use~\(\bar{t}\) for the family of occurrences
corresponding to the \emph{top level} term \(t\), i.e.~seen as a set of
terms instead of term occurrences we have \(\bar{t} \subseteq  \sub(t)\). So for
any term occurrence \(s\), \(\bar{s}\) is not necessarily the full
family of \(s\) in the complete proof tree as \(s\) could be a subterm
of the top level term \(t\) of the family. For any occurrence~\(s\) in a
sequent \(S\) of the proof tree though, \(\bar{s}\) is the family of
\(s\) relative to the subtree \(T\uhp S\) as all related terms in the
premises of G3lp rules are subterms of the related term in the
conclusion.

We also see that most rules of G3lp only relate terms to each other used
for the same subformula \(A\). The two exceptions are the \((\supset  \cdot )\) rule
and the \((\supset  !)\) rule. Similar to the cut rules from the previous
section, \((\supset  \cdot )\) relates subformulas and symbols of different
polarities as well as terms used for different formulas. So we will use
the same approach to define prehistoric relations of term families for
any polarity:

\begin{definition}[Prehistoric Relation in G3lp] A family
\(\bar{t_i}\) has a \emph{prehistoric relation} to another family
\(\bar{t_j}\), in notation \(i \prec  j\), if there is a \((\supset  :)\) rule
introducing an occurrence belonging to \(\bar{t_j}\) with premise \(S\),
such that there is an occurrence belonging to \(\bar{t_i}\) in \(S\).
\end{definition}

Given that we now have defined families of terms and prehistoric
relations between them in G3lp, it is interesting to see what happens
with this relations if we look at the forgetful projection of a G3lp
proof. That is, what happens on the G3s side if we construct a proof
tree with the forgetful projections of the original sequents. Of course
we do not get a pure G3s proof as most of the G3lp rules have no direct
equivalent in G3s. We will therefore define new rules, which are the
forgetful projection of a G3lp rule denoted for example by \((\supset  !)\GassCircli  \)
for the forgetful projection of a \((\supset  !)\) rule. The following two
lemmas show that all this new rules are admissible in G3s + \((\GassBox \Cut)\).

\begin{lemma} \label{boxbox} \(\Glp \vdash  \Gamma  \supset  \Delta , \GassBox A\) iff
\(\Glp \vdash  \Gamma  \supset  \Delta , \GassBox \GassBox A\). \end{lemma}

\begin{proof}

The (\GassLeft ) direction is just inversion for \((\supset  \GassBox )\). The (\GassRight ) direction is
proven by the following structural induction:

1.~case: \(\GassBox A\) is a weakening formula of the last rule. Just weaken in
\(\GassBox \GassBox A\).

2.~case: \(\GassBox A\) is a side formula of the last rule. Use the induction
hypothesis on the premises and append the same last rule.

3.~case: \(\GassBox A\) is the principal formula of the last rule. Then the last
rule is a \((\supset  \GassBox )\) rule and has the following form:

\AXC{$\GassBox \Gamma  \supset  A$}
\RightLabel{$(\supset  \GassBox )$}
\UIC{$\Gamma ', \GassBox \Gamma  \supset  \Delta , \GassBox A$}
\DP

Use an additional \((\supset  \GassBox )\) rule to get the necessary proof as follows:

\AXC{$\GassBox \Gamma  \supset  A$}
\RightLabel{$(\supset  \GassBox )$}
\UIC{$\GassBox \Gamma  \supset  \GassBox A$}
\RightLabel{$(\supset  \GassBox )$}
\UIC{$\Gamma ', \GassBox \Gamma  \supset  \Delta , \GassBox \GassBox A$}
\DP
\end{proof}

\begin{lemma}

The forgetful projection of all rules in G3lp are admissible in G3s +
(\GassBox Cut). \end{lemma}

\begin{proof}

The subset G3c is shared by G3lp and G3s and is therefore trivially
admissible. The forgetful projection of the rule \((\supset  +)\) is just a
contraction and therefore also admissible. The forgetful projection of
the rules \((\supset  :)_t\) and \((\supset  :)_c\) are \((\supset  \GassBox )\) rules in G3s. The
forgetful projection of \((\supset  \cdot )\) is a \((\GassBox \Cut)\). Finally the
forgetful projection of a \((\supset  !)\) rule has the following form:

\AXC{$\Gamma  \supset  \Delta , \GassBox A, \GassBox \GassBox A$}
\RightLabel{$(\supset  !)\GassCircli  $}
\UIC{$\Gamma  \supset  \Delta , \GassBox \GassBox A$}
\DP

That rule is admissible by lemma \ref{boxbox} and a contraction.
\end{proof}

Instead of working with a G3s system with all this extra rules included,
we will define a forgetful projection from a G3lp proof to a G3s +
\((\GassBox \Cut)\) proof by eliminating all  contractions using the
algorithm implicitly defined in the proof of contraction admissibility 
and eliminating the \((\supset  !)\GassCircli  \) rules by the algorithm
implicitly described in the proof for lemma \ref{boxbox}.

For the following lemmas and proofs we fix an arbitrary G3lp proof
\(\GassT  = (T, R)\) and its forgetful projection \(\GassT \GassCircli   = (T', R')\) as defined
below.

\begin{definition}[Forgetful Projection of a G3lp Proof] The
forgetful projection of a G3lp proof \(\GassT  = (T, R)\) for a LP sequent
\(\Gamma  \supset  \Delta \) is the G3s + \((\GassBox \Cut)\) proof \(\GassT \GassCircli   = (T', R')\) for
\(\Gamma \GassCircli   \supset  \Delta \GassCircli  \) inductively defined as follows:

1. case: The last rule of \(\GassT \) is an axiom. Then \(\GassT \GassCircli  \) is just
\(\Gamma \GassCircli   \supset  \Delta \GassCircli  \) which is an axiom of G3s.

2. case: The last rule of \(\GassT \) is a \((\supset  \to )\) or a \((\to  \supset )\) rule with
premises \(S_i\). Then \(\GassT \GassCircli  \) has the same last rule with \((\GassT \uhp S_i)\GassCircli  \)
as proofs for the premises \(S_i\GassCircli  \).

3. case: The last rule of \(\GassT \) is a \((\supset  :)_c\) or \((\supset  :)_t\) rule
with premise \(S\). Then \(\GassT \GassCircli  \) has a \((\supset  \GassBox )\) as last rule with
\((\GassT \uhp S)\GassCircli  \) as proof for the premise \(S\GassCircli  \).

4. case: The last rule of \(\GassT \) is a \((\supset  +)\) rule with premise \(S\).
Then \(\GassT \GassCircli  \) is \((\GassT \uhp S)\GassCircli  \) with the necessary contraction applied.

5. case: The last rule of \(\GassT \) is a \((\supset  \cdot )\) rule with premises
\(S_0\) and \(S_1\). Then \(\GassT \GassCircli  \) has a \((\GassBox \Cut)\) as last rule with
\((\GassT \uhp S_i)\GassCircli  \) as proofs for the premises \(S_i\GassCircli  \).

6. case: The last rule of \(\GassT \) is a \((\supset  !)\) rule with premise \(S\).
Then we get a G3s + \((\GassBox \Cut)\) proof for \(\Gamma \GassCircli   \supset  \Delta \GassCircli  , \GassBox \GassBox A\) from the
proof \((\GassT \uhp S)\GassCircli  \) by lemma \ref{boxbox}. \(\GassT \GassCircli  \) is that proof with the
additional \(\GassBox \GassBox A\) removed by contraction as \(\GassBox \GassBox A \in  \Delta \GassCircli  \).

\end{definition}

To reason about the relations between a G3lp proof \(\GassT \) and its
forgetful projection \(\GassT \GassCircli  \), the following algorithm to construct \(\GassT \GassCircli  \)
is useful:

\begin{enumerate}
\def\labelenumi{\arabic{enumi}.}
\item
  Replace all sequents by their forgetful projection.
\item
  Add the additional \((\supset  \GassBox )\) rules and prepend additional \(\GassBox \) where
  necessary, so that the forgetful projections of \((\supset  !)\) reduce to
  simple contractions.
\item
  Eliminate all contractions to get a G3s + \((\GassBox \Cut)\) proof.
\end{enumerate}

It is not immediately clear that contracting formulas only removes
occurrences as the proof uses inversion which in turn also adds
weakening formulas. But all the deconstructed parts weakened in this way
get contracted again in the next step of the contraction. In the end the
contracted proof tree is always a subset of the original proof tree.

That means that also \(\GassT \GassCircli  \) is a subset of the tree constructed in step
2 of the algorithm. From this we see that all \(\GassBox \) occurrences in
\(\GassT \GassCircli  \) have a term occurrence in \(\GassT \) mapped to them if we consider the
extra \(\GassBox \) occurrences introduced in step~2 (resp.~in case 6 of the
definition) as replacements of the same term as the \(\GassBox \) occurrences
they are contracted with and also consider the extra sequents
\(\GassBox \Gamma  \supset  \GassBox A\) introduced in step 2 as copies of the same formulas in the
original sequent \(\Gamma ', \GassBox \Gamma  \supset  \Delta , \GassBox A\) derived by the original \((\supset  \GassBox )\)
rule.

\begin{lemma}

For any family \(f_i\) of \(\GassBox \) occurrences in \(\GassT \GassCircli  \) there is a unique
proof term family \(\bar{t}_{\tilde{i}}\) in \(\GassT \) such that
\(s \in  \bar{t}_{\tilde{i}}\) for all proof term occurrences \(s\) mapped
to \(\GassBox \) occurrences in \(f_i\). \end{lemma}

\begin{proof}

For any two directly corresponding \(\GassBox \) occurrences we show that the
two mapped term occurrences correspond directly or by reflexive closure:

1.~case: The two \(\GassBox \) occurrences are added in step 2 of the algorithm.
Then the mapped term occurrences are the same occurrence and correspond
by reflexive closure.

2.~case: The two \(\GassBox \) occurrence correspond directly by a rule which is
the forgetful projection of a rule in \(\GassT \). Then the mapped term
occurrences also correspond as all G3lp rules with a direct equivalent
in G3s have the same correspondences. Notice that lemma \ref{boxbox}
only removes weakening formulas from existing \((\supset  \GassBox )\) rules. So this
still holds for \((\supset  \GassBox )\) rules and their corresponding \((\supset  :)\) rules
even after applying lemma \ref{boxbox}.

3.~case: The two \(\GassBox \) occurrences correspond directly by a \((\supset  \GassBox )\)
rule added in step 2 of the algorithm. Then the rule together with the
previous rule has the following form:

\AXC{$\GassBox \Gamma  \supset  A$}
\RightLabel{$(\supset  \GassBox )$}
\UIC{$\GassBox \Gamma  \supset  \GassBox A$}
\RightLabel{$(\supset  \GassBox )$}
\UIC{$\Gamma ', \GassBox \Gamma  \supset  \Delta , \GassBox \GassBox A$}
\DP

As the formulas in \(\GassBox \Gamma  \supset  \GassBox A\) are considered copies of the original
sequent \(\Gamma ', \GassBox \Gamma  \supset  \Delta , \GassBox A\), and the sequent \(\Gamma ', \GassBox \Gamma  \supset  \Delta , \GassBox \GassBox A\) is
considered the same sequent with an additional \(\GassBox \) symbol, the mapped
term occurrences are actually the same and therefore correspond by
reflexive closure.

As direct correspondence in the G3s proof is a subset of correspondence
in the G3lp proof, so is its transitive and reflexive closure. So for
any two corresponding \(\GassBox \) occurrences of a family \(f_i\) the mapped
term occurrences also correspond and therefore belong to the same family
\(\bar{t}_{\tilde{i}}\). \end{proof}

\begin{lemma}

If \(i \prec  j\) in \(\GassT \GassCircli  \) then either \(\tilde{i} = \tilde{j}\) or
\(\tilde{i} \prec  \tilde{j}\) in \(\GassT \) for the term families
\(\bar{t}_{\tilde{i}}\) and \(\bar{t}_{\tilde{j}}\) from the previous
lemma. \end{lemma}

\begin{proof}

\(i \prec  j\) in \(\GassT \GassCircli  \), so there is a \((\supset  \GassBox )\) rule in \(\GassT \GassCircli  \) introducing
an occurrence \(\GassBox _j\) of \(f_j\) with an occurrence \(\GassBox _i\) of \(f_i\)
in the premise. For the mapped term occurrences \(s_i\) and \(s_j\) in
\(\GassT \) we have \(s_i \in  \bar{t}_{\tilde{i}}\) and
\(s_j \in  \bar{t}_{\tilde{j}}\) by the previous lemma. From this it
follows that \(\tilde{i} \prec  \tilde{j}\) or \(\tilde{i} = \tilde{j}\) by
an induction on the proof height:

1.~case: The \((\supset  \GassBox )\) rule is the forgetful projection of a \((\supset  :)\)
rule. Then we have \(\tilde{i} \prec  \tilde{j}\) directly by the definition
of prehistoric relations for G3lp proofs using the occurrences \(s_i\)
in the premise of the rule \((\supset  :)\) introducing the occurrence \(s_j\).

2.~case: The \((\supset  \GassBox )\) rule is added in step 2 of the algorithm. Then
the rule together with the previous rule has the following form:

\AXC{$\GassBox \Gamma  \supset  A$}
\RightLabel{$(\supset  \GassBox )$}
\UIC{$\GassBox \Gamma  \supset  \GassBox _kA$}
\RightLabel{$(\supset  \GassBox )$}
\UIC{$\Gamma ', \GassBox \Gamma  \supset  \Delta , \GassBox _j\GassBox _kA$}
\DP

For the term occurrence \(s_k\) mapped to the occurrence \(\GassBox _k\) we have
\(s_j = !s_k\) and \(s_k \in  \bar{t}_{\tilde{j}}\) as \(s_j\) is the top
level term of the principal formula of a \((\supset  !)\) rule. If the
occurrence \(\GassBox _i\) is the occurrence \(\GassBox _k\) then
\(\tilde{i} = \tilde{j}\) and we are finished. If the occurrence \(\GassBox _i\)
is not the occurrence \(\GassBox _k\) then there is a corresponding occurrence
\(\GassBox '_i\) with a corresponding mapped term \(s'_i\) in the sequent
\(\GassBox \Gamma  \supset  A\) and we have \(i \prec  k\) from the previous \((\supset  \GassBox )\). As
\(\bar{t}_{\tilde{j}}\) is also the term family of \(s_k\) we get
\(\tilde{i} \prec  \tilde{j}\) or \(\tilde{i} = \tilde{j}\) by induction
hypothesis on the shorter proof up to the that \((\supset  \GassBox )\) rule with the
occurrences \(\GassBox '_i\), \(s'_i\), \(\GassBox _k\) and \(s_k\). \end{proof}

\begin{lcorollary} \label{forgetful} If \(\GassT \) is prehistoric-cycle-free
then also \(\GassT \GassCircli  \) is prehistoric-cycle-free. \end{lcorollary}

\begin{proof}

The contraposition follows directly from the lemma as for any cycle
\[i_0 \prec  \cdots \prec  i_n \prec  i_0 \text{ in } \GassT \GassCircli  \] we get a cycle in \(\GassT \) by removing
duplicates in the list \(\tilde{i}_0, \ldots, \tilde{i}_n\) of mapped term
families \(\bar{t}_{\tilde{i}_0}, \ldots \bar{t}_{\tilde{i}_n}\).
\end{proof}

We will now come back to our example formula \(\lnot \GassBox (P \land  \lnot \GassBox P)\) from
section~\ref{self-referentiality}. Figure \ref{g3lpproof} contains a
proof of the same realization \(\lnot x{:}(P \land  \lnot t\cdot x{:}P)\) in G3lp as well as
the forgetful projection of that proof in G3s + (\GassBox Cut). For simplicity
we assumed that \((A \land  B \to  A)\) is an axiom A0 and therefore \(t\) is a
constant.

This proofs display the logical dependencies that make the formula
self-referential in quite a different way than the original G3s proof in
figure \ref{proofs}. There are three families of \(\GassBox \) in the G3s + (\GassBox Cut)
proof. Two are the same families as in the G3s proof, occur in the root
sequent and have a consistent polarity throughout the proof. We therefore
simply use the symbols \(\boxplus \) and \(\boxminus \) for this families. The third one
is part of the cut formula and therefore does not occur in the final
sequent and does not have consistent polarity throughout the proof. We
use \(\GassBox \) for occurrences of this family in the proof.

All left prehistoric relations of the proof are from left branch of the
cut where we have \(\boxminus  \prec _L \boxplus \) and the cycle \(\boxplus  \prec _L \boxplus \). Other than in
the G3s proof, the two \(\boxplus \) occurrences are used for different formulas
\(P\) and \(P \land  \lnot \GassBox P\) and the connection between the two is established
by the (\GassBox Cut) with \(\GassBox (P \land  \lnot \GassBox P \to  P)\). A similar situation is necessary
for any prehistoric cycle in a G3lp proof as we will show formally.

\afterpage{
\thispagestyle{plain}
\begin{landscape}
\begin{figure} \caption{G3lp proof} \label{g3lpproof}
\vspace{2mm}
\AXC{$P, \lnot t\cdot x{:}P, x{:}(P \land  \lnot t\cdot x{:}P) \supset  P$}
\AXC{$P, t\cdot x{:}P \supset  P$}
\RightLabel{$(: \supset )$}
\UIC{$t\cdot x{:}P \supset  P$}
\RightLabel{$(\supset  :)_t$}
\UIC{$P, t\cdot x{:}P, x{:}(P \land  \lnot t\cdot x{:}P) \supset  t\cdot x{:}P$}
\RightLabel{$(\supset  \lnot )$}
\UIC{$P, x{:}(P \land  \lnot t\cdot x{:}P) \supset  t\cdot x{:}P, \lnot t\cdot x{:}P$}
\RightLabel{$(\lnot  \supset )$}
\UIC{$P, \lnot t\cdot x{:}P, x{:}(P \land  \lnot t\cdot x{:}P) \supset  \lnot t\cdot x{:}P$}
\RightLabel{$(\supset  \land )$}
\BIC{$P, \lnot t\cdot x{:}P, x{:}(P \land  \lnot t\cdot x{:}P) \supset  P \land  \lnot t\cdot x{:}P$}
\RightLabel{$(\land  \supset )$}
\UIC{$P \land  \lnot t\cdot x{:}P, x{:}(P \land  \lnot t\cdot x{:}P) \supset  P \land  \lnot t\cdot x{:}P$}
\RightLabel{$(: \supset )$}
\UIC{$x{:}(P \land  \lnot t\cdot x{:}P) \supset  P \land  \lnot t\cdot x{:}P$}
\RightLabel{$(\supset  :)_t$}
\UIC{$P, x{:}(P \land  \lnot t\cdot x{:}P) \supset  x{:}(P \land  \lnot t\cdot x{:}P), t\cdot x{:}P$}

\AXC{$P, \lnot t\cdot x{:}P \supset  P$}
\RightLabel{$(\land  \supset )$}
\UIC{$P \land  \lnot t\cdot x{:}P \supset  P$}
\RightLabel{$(\supset  \to )$}
\UIC{$ \supset  P \land  \lnot t\cdot x{:}P \to  P$}
\RightLabel{$(\supset  :)_c$}
\UIC{$P, x{:}(P \land  \lnot t\cdot x{:}P) \supset  t{:}(P \land  \lnot t\cdot x{:}P \to  P), t\cdot x{:}P$}

\RightLabel{$(\supset  \cdot )$}
\BIC{$P, x{:}(P \land  \lnot t\cdot x{:}P) \supset  t\cdot x{:}P$}
\RightLabel{$(\lnot  \supset )$}
\UIC{$P, \lnot t\cdot x{:}P, x{:}(P \land  \lnot t\cdot x{:}P) \supset $}
\RightLabel{$(\land  \supset )$}
\UIC{$P \land  \lnot t\cdot x{:}P, x{:}(P \land  \lnot t\cdot x{:}P) \supset $}
\RightLabel{$(: \supset )$}
\UIC{$x{:}(P \land  \lnot t\cdot x{:}P) \supset $}
\RightLabel{$(\supset  \lnot )$}
\UIC{$\supset  \lnot x{:}(P \land  \lnot t\cdot x{:}P)$}
\DP

\vspace{2mm}

\AXC{$P, \lnot \boxplus P, \boxminus (P \land  \lnot \boxplus P) \supset  P$}
\AXC{$P, \GassBox P \supset  P$}
\RightLabel{$(\GassBox  \supset )$}
\UIC{$\GassBox P \supset  P$}
\RightLabel{$(\supset  \GassBox )$}
\UIC{$P, \GassBox P, \boxminus (P \land  \lnot \boxplus P) \supset  \boxplus P$}
\RightLabel{$(\supset  \lnot )$}
\UIC{$P, \boxminus (P \land  \lnot \boxplus P) \supset  \boxplus P, \lnot \GassBox P$}
\RightLabel{$(\lnot  \supset )$}
\UIC{$P, \lnot \boxplus P, \boxminus (P \land  \lnot \boxplus P) \supset  \lnot \GassBox P$}
\RightLabel{$(\supset  \land )$}
\BIC{$P, \lnot \boxplus P, \boxminus (P \land  \lnot \boxplus P) \supset  P \land  \lnot \GassBox P$}
\RightLabel{$(\land  \supset )$}
\UIC{$P \land  \lnot \boxplus P, \boxminus (P \land  \lnot \boxplus P) \supset  P \land  \lnot \GassBox P$}
\RightLabel{$(\GassBox  \supset )$}
\UIC{$\boxminus (P \land  \lnot \boxplus P) \supset  P \land  \lnot \GassBox P$}
\RightLabel{$(\supset  \GassBox )$}
\UIC{$P, \boxminus (P \land  \lnot \boxplus P) \supset  \boxplus (P \land  \lnot \GassBox P), \boxplus P$}

\AXC{$P, \lnot \GassBox P \supset  P$}
\RightLabel{$(\land  \supset )$}
\UIC{$P \land  \lnot \GassBox P \supset  P$}
\RightLabel{$(\supset  \to )$}
\UIC{$ \supset  P \land  \lnot \GassBox P \to  P$}
\RightLabel{$(\supset  \GassBox )$}
\UIC{$P, \boxminus (P \land  \lnot \boxplus P) \supset  \boxplus (P \land  \lnot \GassBox P \to  P), \boxplus P$}

\RightLabel{(\GassBox Cut)}
\BIC{$P, \boxminus (P \land  \lnot \boxplus P) \supset  \boxplus P$}
\RightLabel{$(\lnot  \supset )$}
\UIC{$P, \lnot \boxplus P, \boxminus (P \land  \lnot \boxplus P) \supset $}
\RightLabel{$(\land  \supset )$}
\UIC{$P \land  \lnot \boxplus P, \boxminus (P \land  \lnot \boxplus P) \supset $}
\RightLabel{$(\GassBox  \supset )$}
\UIC{$\boxminus (P \land  \lnot \boxplus P) \supset $}
\RightLabel{$(\supset  \lnot )$}
\UIC{$\supset  \lnot \boxminus (P \land  \lnot \boxplus P)$}
\DP
\end{figure}
\end{landscape}
}

\begin{lemma}

All occurrences belonging to a term family \(\bar{t}\) in a premise
\(S\) of any \((\supset  :)\) rule are occurrences of the top level term \(t\)
itself. \end{lemma}

\begin{proof}

All G3lp rules only relate different terms if they are top level terms
on the right. All occurrences of \(s \in  \bar{t}\) in a premise \(S\) of a
\((\supset  :)\) rule correspond either as part of a strict subformula on the
right or as part of a subformula on the left of the conclusion. A
formula on the left can only correspond to a subformula on the right as
a strict subformula. Therefore all corresponding occurrences of \(s\) on
the right in the remaining path up to the root are part of a strict
subformula and so all corresponding occurrences of \(s\), left or right,
in the remaining path are occurrences of the same term \(s\). As \(t\)
itself is a corresponding occurrence of \(s\) in that path, we get
\(t = s\). \end{proof}

\begin{lcorollary} \label{corollary} If \(i \prec  j\) for two term families
\(\bar{t_i}\) and \(\bar{t_j}\) of a G3lp proof, then there is \((\supset  :)\)
rule introducing an occurrence \(s \in  \bar{t_j}\) in a formula \(s{:}A\)
such that there is an occurrence of \(t_i\) in \(s{:}A\) (as a term, not
as a family, i.e.~the occurrence of \(t_i\) is not necessary in
\(\bar{t_i}\)). \end{lcorollary}

\begin{proof}

Follows directly from the lemma and the definition of prehistoric
relations for G3lp. \end{proof}

The last corollary gives us a close relationship between prehistoric
relations in G3lp and occurrences of terms in \((\supset  :)\) rules. But it
does not differentiate between the two variants \((\supset  :)_c\) and
\((\supset  :)_t\) used for introducing elements from CS and input formulas
\(t{:}A\) . It is therefore necessary to expand the definition of
self-referentiality by considering all basic justifications and not only
the justification constants:

\begin{definition}[Inputs] 
The \emph{inputs} IN of a G3lp proof are
all LP formulas that are the principal formula of a \((\supset  :)_t\) or
\((\supset  :)_c\) rule. \end{definition}

Notice that the used constant specifications CS is a subset of the
inputs IN. The interpretation here is that \((\supset  :)_t\) introduces
arguments to Skolem style functions by proving the trivial identity
function \(t{:}A \to  t{:}A\). So we have two different clearly marked
sources of basic proofs in G3lp, on the one hand there are the constants
justifying known axioms, on the other hand there are presupposed
existing proofs or arguments to proof functions. Based on this expanded
notion, we can also expand the definition of self-referentiality to
input sets:

\begin{definition}[Self-Referential Inputs]
A input set IN is 
\begin{itemize}
\item
  \emph{directly self-referential} if there is a term
  \(t\) such that \(t{:}A(t) \in  \IN\).
\item
  \emph{self-referential} if there is a subset
  \(A \subseteq  \IN\) such that
  \[A := \{t_0{:}A(t_1), \ldots, t_{n-1}{:}A(t_0).\]
\end{itemize}
\end{definition}

With this definitions we finally arrive at our main result, a counterpart to Yu's
theorem.

\begin{theorem}
If the input set IN of a G3lp proof is non-self-referential, then the
proof is prehistoric-cycle-free. 
\end{theorem}
\begin{proof}
We show the contraposition. Assume there is a prehistoric cycle
\[i_0 \prec  i_1 \prec  \cdots \prec  i_{n-1} \prec  i_0.\] By corollary \ref{corollary}
there exists formulas \(s_k{:}A_k\) in IN such that
\(t_{i_{k}} \in  \sub(A_k)\) and \(s_k \in  \sub(t_{i_{k'}})\) with
\(k' := k + 1 \mod n\). From the latter and
\(t_{i_{k' }} \in  \sub(A_{k'})\) follows \(s_k \in  \sub(A_{k' })\). So
\(\{s_k{:}A_k\ | 0 \leq k < n\} \subseteq  \IN\) is a self-referential subset of IN.
\end{proof}

\begin{corollary}
The forgetful projection \(A\GassCircli  \) of an LP formula \(A\) provable with a
non-self-referential input set IN is provable prehistoric-cycle-free in
G3s. \end{corollary}

\begin{proof}
Suppose that \(\GassT \) is a proof of \(A\) from non-self-referential inputs IN. Then
\(\GassT \) is prehistoric-cycle-free as proven above. So by corollary
\ref{forgetful} \(\GassT \GassCircli  \) is a prehistoric-cycle-free proof of \(A\GassCircli  \) in
\(\Gs + (\GassBox \Cut)\). Finally there is a prehistoric-cycle-free proof of
\(A\GassCircli  \) in G3s by corollary \ref{boxcutcycle}. \end{proof}

\subsection{Counterexample}\label{counterexample}

The main result of the last section does not exactly match Yu's result.
We have shown that prehistoric cycles in G3s are sufficient for
self-referentiality but only for the expanded definition of
self-referentiality considering the set of all inputs IN. The question
arises if this expansion is actually necessary. The following
counterexample shows that indeed, prehistoric cycles in G3s are not
sufficient for needing a self-referential CS.

\begin{lemma}

The S4 formula \(A \equiv  \GassBox (P \land  \lnot \GassBox P \to  P) \to  \lnot \GassBox (P \land  \lnot \GassBox P)\) has a realization in
\(\LPG_0\). \end{lemma}

\begin{proof}

Set \(A^r \equiv  y{:}(P \land  \lnot y\cdot x{:}P \to  P) \to  \lnot x{:}(P \land  \lnot y\cdot x{:}P)\). We have
\[y{:}(P \land  \lnot y\cdot x{:}P \to  P) \vdash _{\LPG_0} \lnot x{:}(P \land  \lnot y\cdot x{:}P)\] by the same
derivation as for \(\LP \vdash  \lnot x{:}(P \land  \lnot t\cdot x{:}P)\) replacing the
introduction of \(t{:}(P \land  \lnot t\cdot x{:}P \to  P)\) by the assumption
\(y{:}(P \land  \lnot y\cdot x{:}P \to  P)\) and \(t\) by \(y\). So by the deduction
theorem
\(\LPG_0 \vdash  y{:}(P \land  \lnot y\cdot x{:}P \to  P) \to  \lnot x{:}(P \land  \lnot y\cdot x{:}P)\).\footnote{If
  we assume that \(P \land  \lnot y\cdot x{:}P \to  P\) is an axiom A0, this matches the
  more general result in corollary 7.2 in Artemov~\cite{artemov2001}: \(\LP(\CS) \vdash  F\) if
  and only if \(\LPG_0 \vdash  \CS \supset  F\).} \end{proof}

\begin{lemma}

The S4 formula \(\GassBox (P \land  \lnot \GassBox P \to  P) \to  \lnot \GassBox (P \land  \lnot \GassBox P)\) has no
prehistoric-cycle-free proof. \end{lemma}

\begin{proof}

By inversion for G3s in one direction and an easy deduction in the
other, we have \[G3s \vdash  \supset  \GassBox (P \land  \lnot \GassBox P \to  P) \to  \lnot \GassBox (P \land  \lnot \GassBox P)\] iff
\[G3s \vdash  \GassBox (P \land  \lnot \GassBox P \to  P), \GassBox (P \land  \lnot \GassBox P) \supset. \] In both directions the proofs
remain prehistoric-cycle-free if the other proof was
prehistoric-cycle-free. For a proof of \(\GassBox (P \land  \lnot \GassBox P \to  P), \GassBox (P \land \lnot \GassBox P) \supset \)
we have two possibilities for the last rule:

1.~case: The last rule is a \((\GassBox  \supset )\) rule with \(\GassBox (P \land  \lnot \GassBox P \to  P)\) as
the principal formula. Then the following proof tree shows that we need
a proof for the sequent \(P, \GassBox (P \land  \lnot \GassBox P \to  P), \GassBox (P \land \lnot \GassBox P) \supset \) which is just
the original sequent weakened by \(P\) on the left:

\noindent\makebox[\textwidth]{
\AXC{$P \land \lnot \GassBox P, \GassBox (P \land  \lnot \GassBox P \to  P), \GassBox (P \land \lnot \GassBox P) \supset  P \land \lnot \GassBox P$}
\RightLabel{$(\GassBox  \supset )$}
\UIC{$\GassBox (P \land  \lnot \GassBox P \to  P), \GassBox (P \land \lnot \GassBox P) \supset  P \land \lnot \GassBox P$}
\AXC{$P, \GassBox (P \land  \lnot \GassBox P \to  P), \GassBox (P \land \lnot \GassBox P) \supset $}
\RightLabel{$(\to \supset )$}
\BIC{$P \land  \lnot \GassBox P \to  P, \GassBox (P \land  \lnot \GassBox P \to  P), \GassBox (P \land \lnot \GassBox P) \supset $}
\RightLabel{$(\GassBox  \supset )$}
\UIC{$\GassBox (P \land  \lnot \GassBox P \to  P), \GassBox (P \land \lnot \GassBox P) \supset $}
\DP
}

So for the remaining of the proof we will have to check if weakening
\(P\) on the left helps to construct a prehistoric-cycle-free proof.

2.~case: The last rule is a \((\GassBox  \supset )\) rule with \(\GassBox (P \land  \lnot \GassBox P)\) as the
principal formula. We get as premise the sequent
\[P \land \lnot \GassBox P, \GassBox (P \land  \lnot \GassBox P \to  P), \GassBox (P \land \lnot \GassBox P) \supset, \] which again by inversion and an
easy deduction is provable prehistoric-cycle-free iff
\(P, \GassBox (P \land  \lnot \GassBox P \to  P), \GassBox (P \land \lnot \GassBox P) \supset  \GassBox P\) is provable
prehistoric-cycle-free. It is clear that using \((\GassBox  \supset )\) rules on this
sequent just adds additional copies of the existing formulas by the same
arguments. So by contraction if there is a prehistoric-cycle-free proof
for this sequent, then there is also one ending in a \((\supset  \GassBox )\) rule. The
premise of this rule has to have the form \[\GassBox (P \land  \lnot \GassBox P \to  P) \supset  P\] to
avoid a prehistoric cycle. But the following Kripke model shows that
\[\GassBox (P \land  \lnot \GassBox P \to  P) \to  P\] is not a theorem of S4 and therefore not provable
at all: \[W := {w}, val(P) := \emptyset, R := \{(w, w)\}.\] We have
\(w \Vdash P \land  \lnot \GassBox P \to  P\) because \(w \Vdash \lnot P\) and therefore also
\[w \Vdash \lnot (P \land  \lnot \GassBox P).\] As \(w\) is the only world we get
\(w \Vdash \GassBox (P \land  \lnot \GassBox P \to  P)\) which leads to the final
\(w \Vdash \lnot (\GassBox (P \land  \lnot \GassBox P \to  P) \to  P)\) again because \(w \Vdash \lnot P\).

As all possibilities for a prehistoric-cycle-free proof of
\[\GassBox (P \land  \lnot \GassBox P \to  P), \GassBox (P \land \lnot \GassBox P) \supset \] are exhausted, there is no such proof
and therefore also no prehistoric-cycle-free proof of
\(\supset  \GassBox (P \land  \lnot \GassBox P \to  P) \to  \lnot \GassBox (P \land \lnot \GassBox P)\) \end{proof}

\begin{theorem}

There exists a S4-theorem \(A\) and a LP-formula \(B\) such that \(A\)
has no prehistoric-cycle-free G3s-proof, \(B^\circ = A\) and
\(\LP(\CS^\odot ) \vdash  B\) \end{theorem}

\begin{proof}

\(A := \GassBox (P \land  \lnot \GassBox P \to  P) \to  \lnot \GassBox (P \land  \lnot \GassBox P)\) is a theorem of S4, as
\[\lnot \GassBox (P \land  \lnot \GassBox P)\] already is a theorem of S4. By the previous lemma, there
is no prehistoric-cycle-free proof for \(A\) and by the first lemma
\[B := y{:}(P \land  \lnot y\cdot x{:}P \to  P) \to  \lnot x{:}(P \land  \lnot y\cdot x{:}P)\] is a realization
of \(A\) provable in \(\LP_0\) and therefor also in \(\LP(\CS^\odot )\).
\end{proof}

Finally the question arises if prehistoric cycles are also a necessary
condition on self-referential S4 theorems under the expanded definition.
For this it is necessary to clarify the term inputs for Hilbert style
proofs used in the original definition of LP and in the realization
theorem (thm.~\ref{realization}) as there is no direct equivalent for
\((\supset  :)_t\) rules in the Hilbert style LP calculus as there is for
\((\supset  :)_c\) rules. Looking at the adequacy proof for G3lp, \((\supset  :)_t\)
is used only for the base cases \(A \supset  A\) in proofing axioms of LP. In
the other direction, a \((\supset  :)_t\) rule is translated first to the
trivial proof for \(t{:}A \vdash _{\LP} t{:}A\), but the usage of deduction
theorem could change that to a different proof for example for
\(\vdash _{\LP} t{:}A \to  t{:}A\).

So far, the situation seems pretty clear cut, and we have inputs as
assumptions or as subformulas with negative polarity of formulas proven
by the deduction theorem. This also matches the notion that \((\supset  :)_t\)
rules introduce the arguments of Skolem functions used in the LP
realization. Unfortunately the deductions as constructed in the
deduction theorem sometimes use existing formulas with swapped
polarities. That is, in a deduction constructed by the deduction
theorem, subformulas can occur with negative polarity which only
occurred with positive polarity in the original deduction. Moreover
formulas can be necessary to derive the final formula without occurring
in that formula. So there is no guarantee that all necessary inputs
actually occur in the final formula or that a formula occurring with
negative polarity somewhere in the proof is an input.

So we have no clear definition of inputs in the original definition of
LP matching the definition of inputs in G3lp, and therefore also
currently no way to expand Yu's result to all inputs. But we can
stipulate that the inputs of a derivation \(d\) as constructed by the
realization theorem are exactly the
realizations of formulas \(\boxminus _iA\) with negative polarity in the original
G3s proof. As G3s enjoys the subformula property, that means all inputs
used in the proof thus constructed are actually also inputs in the final
formula of the proof, a property which does not necessarily hold for all
derivations as discussed above. We have to assume without proof that
this definition of inputs somehow matches the exact definition given in
the context of G3lp proofs. That is, there exists a G3lp proof for a G3s
proof where only realizations of formulas with negative polarity are
introduced by \((\supset  :)_t\). Given this stipulations and assumptions, the
following sketch of a proof tries to argue for the necessity of
prehistoric cycles for the expanded definition of self-referentiality:

\begin{conjecture}

If a S4-theorem \(A\) has a left-prehistoric-cycle-free G3s-proof, then
there is a LP-formula \(B\) s.t. \(B^\circ = A\) and \(\LP(\IN^\odot ) \vdash  B\).
\end{conjecture}

\emph{Proof idea.}
Given a left-prehistoric-cycle-free G3s-proof \(\GassT  = (T, R)\) for \(A\),
use the realization theorem to construct a realization function \(r_T^N\) and a
constant specification \(\CS^N\) such that \(B := r_T^N(\an_T(A))\) is a
realization of \(A\) and \(\LP \vdash  B\) by the constructed deduction \(d\).
To simplify the following, we do not enforce a injective constant
specification here and allow multiple proof constants for the same
formula. From this it follows that any constant \(c_{i,j,k}\) is
exclusively used when handling the \((\supset  \GassBox )\) rule \(R_{i,j}\).

Assume for a contradiction that the set of inputs IN used for \(d\) is
self-referential. That is there is a subset
\(\{t_0{:}A_0(t_1), \ldots, t_{n-1}{:}A_{n-1}(t_0)\}\) of IN. The
occurrences of \(t_{k+1 \mod n}\) in \(t_k{:}A_i\) have to be a subterm
of realization term for a principal family \(i_k\) as the construction
of such realization terms are the only place where the constants and
variables of IN can get reused. For every consecutive pair of principal
families \(i_k\) and \(i_{k'}\) thus given, there is a constant or
variable \(t_{k'}\) such that \(t_{k'}\) occurs in the realization term
for \(i_k\) and there is a subterm of the realization term for
\(i_{k'}\) occurring in \(t_{k'}{:}A_{k'} \in  \IN\). We distinguish the
following cases:

1.~case: \(t_{k'}\) is a variable \(x_j\). Then the formula
\(t_{k'}{:}A_{k'}\) is the realization of an annotated S4 formula
\(\boxminus _jA(\boxplus _{i_{k'}})\). That formula occurs on the left of a \((\supset  \GassBox )\)
rule introducing an occurrence of \(\boxplus _k\) as \(x_j\) is in the
realization term of \(\boxplus _k\). Therefore we have \(i_{k'} \prec  i_k\).

2.~case: \(t_{k'}\) is a constant \(c_{j,l,m}\). Then the formula
\(t_{k'}{:}A_{k'}\) is added to the CS when handling a \((\supset  \GassBox )\) rule
\(R_{j,l}\) introducing an occurrence of \(\boxplus _j\). \(c_{j,l,m}\) is in
the realization term of \(\boxplus _k\) so \(R_{j,l}\) lies in a prehistory of
\(\boxplus _k\). At the same time, the term \(t_{k'}\) occurs in the formula
\(c_{j,l,m}{:}A_{k'}\) as part of a term \(t\) used in the construction
of the realization of \(\boxplus _{k'}\). As \(c_{j,l,m}{:}A_{k'}\) is
introduced when realizing \(R_{j,l}\), \(A_{k'}\) occurs in the proof of
the premise and there has to be an occurrence of \(\boxplus _{k'}\) in the
prehistory of \(R_{j,l}\). Together we get that \(\boxplus _{k'}\) occurs in a
prehistory of \(\boxplus _k\) and therefore \(i_{k'} \prec  i_k\) by lemma
\ref{global}.-

So for all \(k < n\) we get \(i_{k'} \prec  i_k\) and the list of principal
families \(i_0, \ldots, i_{n-1}\) is therefore a prehistoric cycle in
\(\GassT \). 

\section{Conclusion}\label{conclusion}

We defined prehistoric relations for Gentzen systems with cut rules and finally for
a Gentzen system G3lp for the logic of proofs LP. This makes it possible to study
prehistoric relations directly in LP and leads to a negative answer on
Yu's conjecture that prehistoric cycles are sufficient for
self-referential S4 theorems. It also leads to an expanded definition of
self-referentiality considering all inputs used to construct
justification terms. With that expanded definition of
self-referentiality prehistoric cycles are \emph{sufficient} for
self-referential theorems in S4, which is the main result of this paper.

Given this expansion, the question goes back to the other direction.
That is, are prehistoric cycles also necessary for the expanded
definition of self-referentiality? Unfortunately this question is not
easy to answer, as already transferring the definitions of inputs to the
original Hilbert style calculus poses problems. A more detailed
discussion of Skolem style functions and their role in LP realizations
will hopefully help to clear this up. It is possible that the definition
of input variables relative to a subformula occurrence and the machinery
used to work with input variables in~\cite{studer} already provides a part of
the answer.

Yu~\cite{yu2014} expanded his result to modal logics T and K4 and their
justification counterparts.
Another open question is whether the same generalization can be done
with the results of this paper. That is, if there are Gentzen style
systems without structural rules for T and K4 together with a consistent
definition of term correspondence and prehistoric relations and a
translation to some variant of G3s.

\par\bigskip
\noindent
\textbf{Acknowledgments}

\noindent
This work was supported by the  grant 200021\_165549 of the Swiss National Science Foundation.


\end{document}